\documentclass[a4paper,10pt]{amsart}
\usepackage[left=2.5cm,right=2.5cm,top=3cm,bottom=3cm]{geometry}
\usepackage{amsthm,amssymb,amsmath,amsfonts,mathrsfs,amscd}
\usepackage[utf8x]{inputenc}
\usepackage[all]{xy}
\usepackage{latexsym}
\usepackage{longtable}
\usepackage{hyperref}
\usepackage{url}

\newfont{\cyr}{wncyr10 scaled 1100}
\newfont{\cyrr}{wncyr9 scaled 1000}

\theoremstyle{plain}
\newtheorem{thm}{Theorem}[section]
\newtheorem{prop}[thm]{Proposition}
\newtheorem{lemma}[thm]{Lemma}
\newtheorem{corol}[thm]{Corollary}
\newtheorem*{ThmA}{Theorem A}

\theoremstyle{definition}
\newtheorem{conjecture}[thm]{Conjecture}
\newtheorem{defn}[thm]{Definition}
\newtheorem{assumption}[thm]{Assumption}

\theoremstyle{remark}
\newtheorem{remark}[thm]{Remark}

\newtheorem*{intro-question}{Question}

\newtheorem{example}[thm]{Example}

\newcommand{\A}{\mathbb A}
\newcommand{\Q}{\mathbb Q}

\newcommand{\Z}{\mathbb Z}
\newcommand{\R}{\mathbb R}
\newcommand{\C}{\mathbb C}
\newcommand{\PP}{\mathbb P}
\newcommand{\F}{\mathbb F}
\newcommand{\T}{\mathbb T}

\newcommand{\cR}{\mathcal R}

\DeclareMathOperator{\Pic}{Pic}
\DeclareMathOperator{\End}{End}
\DeclareMathOperator{\Aut}{Aut}
\DeclareMathOperator{\Autm}{Aut^{\mathrm{mod}}}

\DeclareMathOperator{\Emb}{Emb}

\DeclareMathOperator{\Hom}{Hom}

\DeclareMathOperator{\Gal}{Gal}
\DeclareMathOperator{\GL}{GL}

\DeclareMathOperator{\SL}{SL}

\DeclareMathOperator{\M}{M}

\DeclareMathOperator{\Ta}{Ta}

\DeclareMathOperator{\Jac}{Jac}

\newcommand{\Ne}{N_\mathrm{Eic}}
\newcommand{\Nc}{N_\mathrm{Car}}
\newcommand{\val}{\mathrm{val}}

\newcommand{\rec}{\mathrm{rec}}
\newcommand{\op}{\mathrm{op}}

\usepackage[usenames]{color}
\definecolor{Indigo}{rgb}{0.2,0.1,0.7}
\definecolor{Violet}{rgb}{0.5,0.1,0.7}
\definecolor{White}{rgb}{1,1,1}
\definecolor{Green}{rgb}{0.1,0.9,0.2}

\newcommand{\lra}{{\longrightarrow }} 
\newcommand{\longmono}{\mbox{$\lhook\joinrel\longrightarrow$}}

\newcommand{\mat}[4]{\left(\begin{array}{cc}#1&#2\\#3&#4\end{array}\right)}
\newcommand{\smallmat}[4]{\bigl(\begin{smallmatrix}#1&#2\\#3&#4\end{smallmatrix}\bigr)}

\newcommand{\dirlim}{\mathop{\varinjlim}\limits}
\newcommand{\invlim}{\mathop{\varprojlim}\limits}

\newcommand{\mf}[1]{{\mathfrak{#1}}}

\newcommand{\cO}{{\mathcal O}}

\begin{document}

\title{Heegner points on Hijikata--Pizer--Shemanske curves}
\date{}
\author{Matteo Longo \and Victor Rotger \and Carlos de Vera-Piquero}

\begin{abstract} 
We study Heegner points on elliptic curves, or more generally modular abelian varieties, coming from uniformization 
by Shimura curves attached to a rather general type of quaternionic orders closely related to those introduced by 
Hijikata--Pizer--Shemanske in the 80's. We address several questions arising from the Birch and Swinnerton-Dyer (BSD) 
conjecture in this general context. In particular, under mild technical conditions, we show the existence of non-torsion 
Heegner points on elliptic curves in all situations in which the BSD conjecture predicts their existence.    
\end{abstract}

\subjclass[2010]{}

\keywords{}

\thanks{During the work on this paper, M.L. was partially supported by PRIN 2010/11
``Arithmetic Algebraic Geometry and Number Theory" and PRAT 2013 ``Arithmetic of Varieties over Number Fields"; 
C. de V.-P. was partially supported by Grant MTM2012-34611 and by the German Research Council via SFB/TR 45; and 
V. R. was partially supported by Grant MTM2012-34611}

\address{M. L.: Dipartimento di Matematica, Universit\`a di Padova, Padova, Italy.}
\email{mlongo@math.unipd.it}
\urladdr{www.math.unipd.it/~mlongo}

\address{V. R.: Departament de Matem\`atiques, Universitat Polit\`ecnica de Catalunya, Barcelona, Spain.}
\email{victor.rotger@upc.edu}
\urladdr{https://www-ma2.upc.edu/vrotger}

\address{C. de V.-P.: Fakult\"at f\"ur Mathematik, Universit\"at Duisburg-Essen, Essen, Germany.}
\email{cdeverapiquero@gmail.com}
\urladdr{https://www.uni-due.de/~adf538d}

\maketitle


\section*{Introduction}

This work grows on an attempt to study Shimura curves and Heegner points in arithmetically interesting 
situations which at present are poorly understood. 

To motivate our study, let us present an example first. Suppose that $E/\Q$ is an elliptic curve of 
conductor $p^2q$, where $p$ and $q$ are two distinct odd primes, and let $K$ be an imaginary quadratic 
field in which $p$ is ramified and $q$ is inert. In this situation one can not construct rational 
points on $E(K)$ using parametrizations by modular curves $X_0(N)$, because the classical Heegner 
hypothesis fails. 

However, in this scenario it may perfectly be the case that the sign of the functional equation satisfied 
by $L(E/K,s)$ be $-1$ (indeed, this holds under certain arithmetic conditions on the local root numbers of 
the functional equation of $L(E/K,s)$ at $p$ and $q$ which are made precise below).  If in addition the 
order of vanishing of $L(E/K,s)$ at $s=1$ is one, then the Birch and Swinnerton-Dyer conjecture predicts 
the existence of a rational point in $E(K)$ generating $E(K)\otimes\Q$. One expects to construct points of 
infinite order in $E(K)$ using parametrizations by Shimura curves $X_U$ attached to the quaternion algebra 
of discriminant equal to $pq$ and open compact subgroups $U$ of $(B\otimes\hat \Z)^\times$; note that $U$ 
can not be maximal at $p$ because the conductor of the elliptic curve is divisible by $p^2$. 

Nevertheless, in this setting one still has a rich 
theory of local quaternionic orders whose level is divisible by $p^2$ and optimal embeddings of quadratic orders 
of $K$ into them; this theory has been developed by Pizer \cite{Pizer}, and then extended in greater 
generality by Hijikata--Pizer--Shemanske in \cite{HPS1}. Shimura curves attached to these orders play an important role in 
what follows, motivating the title of this paper. 

One of the main motivations that led us to work on this project  is that these curves can be $p$-adically uniformized by the 
$p$-adic rigid analytic space corresponding to the first (abelian) covering of the Drinfel'd tower over the $p$-adic 
upper half plane $\mathcal H_p$. This rigid analytic space has an explicit description (see \cite{Tei}) which can be used 
to study $p$-adic aspects of Heegner points, including their connection to Iwasawa theory and $p$-adic $L$-functions, as in the case where the elliptic curve 
has multiplicative reduction and can be uniformized by Drinfel'd upper half plane (cf. \cite{BDMTcurves}, 
\cite{BD-Inv}, \cite{BD-Duke}). 

It would also be highly interesting to extend the theory of Stark--Heegner 
points in this context (starting with the foundational paper \cite{Dar}, and developed in \cite{BD-Hida}, \cite{BD-Rat},
\cite{Das}, \cite{BDD}, \cite{Green}, \cite{LRV1}, \cite{LRV2}, \cite{LV}). Such a generalization is however not straight-forward, essentially because the jacobian varieties of the Shimura curves referred to above have additive reduction at $p$ (as opposed to having toric reduction, which is a crucial feature in the above approaches). We regard this as an exciting obstacle to overcome rather than a forbidding difficulty, and this  note aims to settle the first step towards this program that we hope to pursue in the near future.

In any case let us stress from the beginning that, independently of any eventual applications of our work building on the $p$-adic uniformization of elliptic curves, in this article we address 
much more general questions related to the existence of Heegner points than those coming from the specific example 
described above (and, therefore, use much more general orders than those introduced in \cite{Pizer} and \cite{HPS1}). Indeed, many of our arguments involving Euler systems and applications to the conjecture of Birch and Swinnerton-Dyer rely on the following apparently na\"ive question:

\begin{intro-question}
Given an elliptic curve $E/\Q$, an imaginary quadratic field $K$ and an anticyclotomic character $\chi$ of 
$\Gal(K^\mathrm{ab}/K)$ factoring through the ring class field $H_c$ of conductor $c\geq1$ of $K$, under which conditions there exist non-trivial Heegner points in $E(H_c)$?
\end{intro-question}

Let us now describe in more detail the main results of this article and their implications to the conjecture of Birch and Swinnerton-Dyer. 

As above, let $E/\Q$ be an elliptic curve of conductor $N$, $K$ be an imaginary quadratic field and 
\[
\chi: G_K = \Gal(\bar K/K) \lra \C^\times
\]
be a character of finite order. 
We assume throughout that $\chi$ is {\em anticyclotomic}, meaning that  
$\chi(\tau \sigma \tau^{-1}) = \chi^{-1}(\sigma)$ for any $\sigma\in G_K$ and 
$\tau \in G_\Q \setminus G_K$. The abelian extension cut out by $\chi$ is then a 
ring class field associated to some order $R_c$ in $K$ of conductor $c=c(\chi)\geq 1$. Let 
$H_c$ denote the corresponding abelian extension, determined by the isomorphism 
$\Gal(H_c/K)\simeq\Pic(R_c)$ induced by the Artin map. 

Let $L(E/K,\chi,s)$ denote the Rankin $L$-series  associated to the twist of $E/K$ by $\chi$. 
Since $\chi$ is anticyclotomic, the motive associated to $L(E/K,\chi,s)$ is Kummer self-dual and 
this implies that the global root number $\varepsilon(E/K,\chi)$ of $L(E/K,\chi,s)$ is either $+1$ 
or $-1$. Assume for the remainder of the article that 
\[
\varepsilon(E/K,\chi) = -1,
\]
hence in particular $L(E/K,\chi,s)$ vanishes at the central critical point $s=1$. Define
\[
(E(H_c)\otimes\C)^\chi := \{x\in E(H_c)\otimes\C: \sigma(x)=\chi(\sigma)x, \forall \sigma\in\Gal(H_c/K)\}.
\]
In this situation, the Galois equivariant version of the Birch--Swinnerton-Dyer conjecture predicts that 
\[
L'(E/K,\chi,1)\neq 0 \stackrel{?}{\Longleftrightarrow} \dim_\C\left(E(H_c)\otimes\C\right)^\chi=1.
\]

In this paper we are concerned with the implication ``$\Rightarrow$''. The well-known strategy to prove this 
implication, after Kolyvagin's fundamental work \cite{Kol}, is to exploit Euler systems of Heegner points on $E$ 
arising from Heegner points on modular curves $X_0(N)$, via a parametrization $$\pi_E: X_0(N)\rightarrow E$$ whose 
existence is a consequence of Wiles's Modularity Theorem \cite{Wi}, \cite{TW}, \cite{BCDT}. More generally, 
in order to enlarge the source of rational points on elliptic curves, one can use uniformizations of our fixed 
elliptic curve $E$ by Jacobians $\Jac(X_{U})$ of Shimura curves $X_{U}$ associated to indefinite quaternion 
algebras $B/\Q$ and compact open subgroups $U\subseteq \hat B^\times=(B\otimes \hat\Z)^\times$. 

When $U=\hat{\mathcal R}^\times = (\mathcal R\otimes\hat\Z)^\times$ for some order $\mathcal R$ of $B$,
we simply write $X_{\mathcal R}$ for $X_{\hat{\mathcal R}^{\times}}$; also, to emphasize the role of $B$, 
we will sometimes write $X_{B,U}$ and $X_{B,\mathcal R}$ for $X_U$ and $X_{\mathcal R}$, respectively. 

Section 1 is devoted to introduce an explicit family of special orders $\cR$ in $B$ which shall play a central 
role in our work. These orders are determined by local data at the primes of bad reduction of the elliptic curve 
$E$, following classical work of Hijikata, Pizer and Shemanske that apparently did not receive the attention it 
justly deserved.

In section 2 we study the Shimura curves $X_{\cR}$ associated to the above mentioned Hijikata--Pizer--Shemanske 
orders and work out explicitly the Jacquet--Langlands correspondence for these curves, which allows 
us to dispose of a rich source of modular parametrizations of the elliptic curve $E$. For any integer $c$, 
there is a (possibly empty) collection of distinguished points on $X_{\mathcal R}$, called \emph{Heegner points 
of conductor $c$}. The set of Heegner points of conductor $c$ on $X_{\mathcal R}$ is in natural correspondence 
with the set of conjugacy classes of optimal embeddings of the quadratic order $R_c$ in the quaternion order 
$\mathcal R$, and we denote it $\mathrm{Heeg}(\mathcal R,K,c)$. We say that a point in $E(H_c)$ is a 
\emph{Heegner point of conductor $c$} associated to the order $\cR$ if it is the image of a Heegner point 
of conductor $c$ in $X_{\mathcal R}$ for some uniformization map 
\[
\pi_E: \Jac(X_{\mathcal R})\rightarrow E
\] 
defined over $\Q$. The corresponding set of Heegner points is denoted $\mathrm{Heeg}_E(\mathcal R,K,c)$.

In section 3 we perform a careful and detailed analysis of the rather delicate and involved theory of optimal 
embeddings of quadratic orders into Hijikata--Pizer--Shemanske orders. Combining all together this allows to 
prove the main result of this article. A slightly simplified version of this result is the following. The main 
virtue of the statement below with respect to previous results available in the literature is that it is both 
{\em general} (removing nearly all unnecessary hypothesis on divisibility and congruence relations among $N$, 
$D$ and $c$) and {\em precise} (pointing out to a completely explicit Shimura curve).

\begin{ThmA}
Let $E/\Q$ be an elliptic curve of conductor $N$ not divisible neither by $2^3$ nor by $3^4$. Let $K$ be an imaginary quadratic field of discriminant $-D$ and $\chi$ be
an arbitrary anticyclotomic character of conductor $c\geq 1$. Assume that $\varepsilon(E/K,\chi)=-1$. Then 
\begin{enumerate}
\item there exists an explicit Hijikata--Pizer--Shemanske order $\cR=\cR(E,K,\chi)$ for which the set of 
Heegner points $\mathrm{Heeg}_E(\mathcal R,K,c)$ in $E(H_c)$ is non-empty. 
\item If $L'(E/K,\chi,1)\neq 0$ and $E$ does not acquire CM over 
any imaginary quadratic field contained in $H_c$, then $\dim_\C\left(E(H_c)\otimes\C\right)^\chi=1$.
\end{enumerate}
\end{ThmA} 

This theorem is proved in the last section of the article, where we also provide a more general statement for 
elliptic curves and prove a similar but weaker result (Theorem \ref{Main_Thm1}) for modular abelian varieties; 
we close the paper with a refined conjecture on the existence of Heegner points on modular abelian varieties.

Theorem A follows from Theorem \ref{main_thm} below, where the condition that $2^3$ and $3^4$ do not divide 
$N$ is considerably relaxed into the much weaker (but more involved) Assumption \ref{assumption2N}. Namely, 
we can also prove the above theorem in all cases where $3^5$ divides $N$; there is only one case with 
$3^4$ dividing $N$ exactly that we can not treat with our arguments (it is a case in which $3$ is inert in 
$K$ and the $p$-power conductor of $\chi$ is equal to $1$). In addition, if $\val_2(N)\geq 4$ is even, then 
the newform $f \in S_2(\Gamma_0(N))$ associated to $E$ by modularity is a twist of a newform 
$f' \in S_2(\Gamma_1(N'))$ for some level $N'$ with $\val_2(N')<\val_2(N)$ (\cite[Theorem 3.9]{HPS2}) and in 
these cases it seems possible to investigate the existence of Heegner points with different methods (namely, 
prove the existence of Heegner points for the abelian variety corresponding to $f'$ and then twisting back to 
show the existence of Heegner points on the elliptic curve). Therefore, if we assume that the form is 
\emph{primitive}, in the sense that it is not a twist of a form of lower level, then we can exclude form our 
discussion all cases in which $\val_2(N)$ is even and greater or equal than $4$. Finally, if we further assume 
the condition that the $2$-component $\pi_{f,2}$ of the automorphic representation $\pi_f$ associated with 
$f$ has minimal conductor among all its twists by quasi-characters of $\Q_2^\times$, then the only cases 
we need to exclude in Theorem A above are the following, where $\Delta$ is the discriminant of the quaternion 
algebra:
\begin{equation}\label{MissingCase}\tag{Missing Cases}
\begin{array}{lc}
\text{$\val_3(N) = 4$, $3\nmid \Delta$, $3$ is inert in $K$ and $\val_3(c)=1$;}\\
\text{$\val_2(N)\geq 3$, $2\nmid \Delta$ and $2$ is ramified in $K$.}\end{array}
\end{equation}
See Corollary \ref{coro-final} for a complete statement. However, let us stress that these missing cases, 
even if they seem to be isolated in the case of elliptic curves,  they are not rare at all in the more general 
context of modular abelian varieties associated with modular forms of level $\Gamma_0(Mp^r)$ with $p$ any 
prime number and $r$ arbitrarily large. 

Statement (2) in the above theorem follows from (1) and well-known Kolyvagin type arguments which are spelled 
out in detail in \cite{Nek}. Namely, given
\begin{itemize}
\item a parametrization of the elliptic curve $E$ by a Shimura curve $X_{B,U}$,
\item a Heegner point $x$ in $X_{B,U}(K^\mathrm{ab})$, rational over a subfield $K(x)\subseteq K^\mathrm{ab}$, and
\item a character $\chi$ factoring through $\Gal(K(x)/K)$,
\end{itemize} 
Nekov\'a\v r shows that if the special value of the derivative of the complex $L$-function at $s=1$ is nonzero, 
then the dimension of the $\C$-vector space $(E(K(x))\otimes\C)^\chi$ is equal to $1$, provided that $E$ does not acquire 
CM over any imaginary quadratic field contained in $K(x)^{\ker(\chi)}$. 

In some sense, our Theorem A \emph{reverses} 
the logical order of the result in \cite{Nek}, starting with a character of a given conductor and asking for a Heegner point 
rational over the subextension of $K^\mathrm{ab}$ cut out by that character. Therefore, the whole focus of our work is on 
statement (1) of the above theorem. More precisely, this work grows out from a systematic study of existence conditions for 
Heegner points in all scenarios in which the BSD conjecture predicts the existence of a non-zero element in 
$(E(H_c)\otimes\C)^\chi$. To understand the flavour of this work, it is therefore important to stress that \emph{we do 
not require any condition on the triplet $(N,D,c)$}, besides the above restrictions at $2$ and $3$ (cf. also Assumption 
\ref{assumption2N}). Quite surprisingly, the interplay between local root numbers, non-vanishing of the first derivative 
of the $L$-function and the theory of optimal embeddings shows that these conditions match perfectly and, in all relevant cases, 
Heegner points do exist.

Our approach consists in three steps. Firstly, we need to find a suitable candidate Shimura curve $X_{\mathcal R}$ equipped with a non-constant map 
$\Jac(X_{\mathcal R})\rightarrow E$. For this, we need to specify the 
ramification set of the quaternion algebra $B$, which is prescribed as usual in terms of local root numbers 
of $\varepsilon(E/K,\chi)$ and $\eta_K(-1)$, where $\eta_K$ is the quadratic character associated with 
$K/\Q$ (see for example \cite{YZZ}). 

Secondly, we need to specify a convenient order $\mathcal R$, and this is
an application of a fine version of the Jacquet--Langlands theory, using works of Pizer, Chen and others. 

Thirdly, having fixed our Shimura curve $X_{\mathcal R}$, we need to prove the existence of Heegner points of 
conductor $c$, and this follows, as hinted above, from a careful study of relations between local root numbers 
and optimal embeddings. The main new ingredient in this context is the adaptation of the 
theory of local orders in quaternion division algebras developed by Hijikata--Pizer--Schemanske.

In closing this introduction, we would like to mention three interesting papers that have recently seen the light, two of them by Kohen--Pacetti \cite{PK}, 
\cite{PK2} and one by Cai--Chen--Liu \cite{CCL}, addressing similar but non-overlapping problems in the case of Cartan level 
structure. As opposed to these works, we rather focus on the case where the modular parametrization of the elliptic curve is given by a Shimura curve associated to an order in a quaternion {\em division} algebra. In addition, it is worth remarking that while \cite{PK}, 
\cite{PK2} and \cite{CCL} study Euler systems of Heegner points in a scenario where these are known to exist, our goal is to show that (at least in the case 
of elliptic curves with local $p$-adic representations not too ramified for $p=2$ and $p=3$) 
it is possible to construct a suitable Shimura curve giving rise to a non-trivial system of Heegner points of prescribed conductor in all situations in which the BSD Conjecture predicts their existence, covering many new cases.

\section{Shimura curves}

Let $\hat{\Z}$ denote the profinite completion of $\Z$, and write $\hat R := R \otimes_{\Z} \hat{\Z}$ 
for every $\Z$-algebra $R$. Fix an integer $\Delta > 1$, which is assumed to be square-free and the 
product of an even number of primes, and let $B$ be the indefinite rational quaternion algebra of reduced 
discriminant $\Delta$. Write $\hat B = B\otimes_{\Q} \hat{\Q}$ for the finite adelization of $B$. We also 
fix a maximal order $\cO$ in $B$; recall that such an order is unique up to conjugation by an element in 
$B^{\times}$. Finally, we shall fix an isomorphism $B_{\infty}:=B\otimes_{\Q}\R \to \M_2(\R)$, under 
which $B^{\times}$ might be seen as a subgroup of $\GL_2(\R)$.

\subsection{Shimura curves}

Let $\mathcal H^{\pm} = \C - \R = \PP^1(\C) - \PP^1(\R)$ be the union of the upper and lower 
complex half planes, which might be identified with the set of $\R$-algebra homomorphisms 
$\Hom(\C,\M_2(\R))$. The action of $B^{\times}$ by linear fractional transformations on $\mathcal H^{\pm}$ 
corresponds under this identification to the action of $B^{\times}$ by conjugation on $\Hom(\C,\M_2(\R))$.

For any compact open subgroup $U$ of $\hat{\cO}^{\times}$, one can consider the topological space of 
double cosets 
\begin{equation}\label{defnXU} 
X_U = \left(U \backslash \hat{B}^{\times}\times \Hom(\C,\M_2(\R))\right)/B^{\times},
\end{equation}
where notice that $U$ acts naturally on $\hat{B}^{\times}$ by multiplication on the left and $B^{\times}$ 
acts both on $\hat{B}^{\times}$ (diagonally) and on $\Hom(\C,\M_2(\R))$. By the work of Shimura and Deligne, 
$X_U$ admits the structure of an algebraic curve over $\Q$ and a canonical model, which we shall still 
denote by $X_U/\Q$. This will be referred to as the Shimura curve associated with $U$.

Although $X_U$ is connected over $\Q$, it might not be in general geometrically connected. Indeed, the set 
of geometric connected components of $X_U$ (that is, the set of connected components of 
$\overline{X}_U := X_U\times_{\Q}\bar{\Q}$) is identified with the finite set of double cosets 
$U \backslash  \hat{B}^{\times} / B^{\times}$. Such components are defined over an abelian extension of 
$\Q$, and via the reciprocity map from class field theory the action of 
$\Gal(\Q^{ab}/\Q) \simeq \hat{\Z}^{\times}$ on them is compatible under the isomorphism  
\[
U \backslash  \hat{B}^{\times} / B^{\times} \, \, \stackrel{\simeq}{\lra} \, \, \mathrm n(U) \backslash \hat{\Q}^{\times} / \Q^{\times} = \mathrm n(U) \backslash \hat{\Z}^{\times}
\]
induced by the reduced norm $\mathrm n$ on $\hat{B}^{\times}$ (by strong approximation) with the natural 
action of $\hat{\Z}^{\times}$ on $\mathrm n(U) \backslash \hat{\Z}^{\times}$.

From the very definition in \eqref{defnXU}, one can naturally define a group of automorphisms of $X_U$, 
which are often called {\em modular}. Namely, if $N(U)$ denotes the normalizer of $U$ in $\hat B^{\times}$, 
then left multiplication by an element $b \in N(U)$ induces an automorphism $\lambda(b): X_U \to X_U$, given 
on points by the rule 
\[
 \lambda(b): [g,f] \, \, \longmapsto \, \, [bg,f].
\]
Here, $[g,f]$ denotes the point on $X_U$ corresponding to a pair $(g,f) \in \hat{B}^{\times}\times \Hom(\C,\M_2(\R))$. 
It is immediate to check that $\lambda(b)$ defines the identity on $X_U$ if and only if $b \in U\Q^{\times}$. 
The group $\Autm(X_U)$ of {\em modular automorphisms} on $X_U$ is then defined to be the group of all the 
automorphisms obtained in this way, so that
\[
 \Autm(X_U) := U\Q^{\times} \backslash N(U).
\]

If $U=\hat{\mathcal S}^{\times}$ is the group of units in the profinite completion of some order 
$\mathcal S \subseteq \cO$, then we shall write $X_{\mathcal S} := X_{\hat{\mathcal S}^{\times}}/\Q$ 
for the {\em Shimura curve associated with the order $\mathcal S$}. In this case, the set of 
geometric connected components of $X_{\mathcal S}$ is identified with the class group $\Pic(\mathcal S)$ 
of $\mathcal S$. 

\begin{remark}
The most common setting in the literature is when $\mathcal S = \mathcal S_{N^+}$ is an Eichler 
order of level $N^+$ in $\cO$, where $N^+\geq 1$ is an integer prime to $N^-:=\Delta(B)$. In this 
case, the Shimura curve $X_{N^+,N^-} := X_{\mathcal S}/\Q$ associated with $\mathcal S$ is not 
only connected but also geometrically connected, and its group $\Autm(X_{N^+,N^-})$ of modular 
automorphisms is the group of Atkin-Lehner involutions, which are indexed by the positive divisors 
of $N^+N^-$. Further, $X_{N^+,N^-}/\Q$ is the coarse moduli space classifying abelian surfaces with 
quaternionic multiplication by $\cO$ and $N^+$-level structure.
 
When $N^-=1$ (so that $B$ is the split quaternion algebra $\M_2(\Q)$), a case which we exclude in 
this paper, the Shimura curve $X_{N^+,1}/\Q$ is the affine modular curve $Y_0(N^+)$ obtained as a 
quotient of the upper half plane by the congruence subgroup $\Gamma_0(N^+)$, whose compactification 
by adding finitely many cusps is the usual modular curve $X_0(N^+)/\Q$.
\end{remark}

In this article, we will be working with Shimura curves associated with certain suborders of $\cO$ 
which are not Eichler orders, but rather with more general orders that for example might have non-trivial 
level at primes dividing $\Delta(B)$. The special class of quaternion order we shall be dealing with is 
described in the next section.

\subsection{Choice of quaternion orders}

Let $p$ be a rational prime and let $B_p$ be a quaternion algebra over $\Q_p$. The object of this 
section is introducing several families of local quaternion orders in $B_p$ which in turn will give 
rise to a fauna of Shimura curves that will serve as the appropriate host of the Heegner systems we 
aim to construct.

Assume first that $B_p=D_p$ is the unique (up to isomorphism) quaternion division algebra over $\Q_p$ 
and let $\cO_p$ be the unique maximal order in $D_p$. If $L_p$ is a quadratic extension of $\Q_p$ and 
$\nu\geq 1$ is an integer, one can define the (local) quaternion order
\[
 R_{\nu}(L_p) = \cO_{L_p} + \pi_p^{\nu-1}\cO_p,
\]
where $\cO_{L_p}$ denotes the ring of integers of $L_p$ and $\pi_p$ is a uniformizer element in $\cO_p$. 
Such local orders are studied in detail by Hijikata, Pizer and Shemanske in \cite{HPS1}. Notice that 
$R_1(L_p)$ coincides with the maximal order $\cO_p$, regardless of the choice of $L_p$. Further, if 
$L_p'$ is another quadratic extension of $\Q_p$ with $L_p \simeq L_p'$, then $R_{\nu}(L_p)$ and $R_{\nu}(L_p')$ 
are conjugated by an element in $D_p^\times$. For $\nu\geq 2$, the order $R_{\nu}(L_p)$ is characterized 
as the unique order in $D_p$ containing $\cO_{L_p}$ and $\pi_p^{\nu-1}\cO_p$ but not containing 
$\pi_p^{\nu-2}\cO_p$. 

\begin{remark}
 If $p$ is odd and $L_p$ is the unique unramified quadratic extension, then $R_{2r+1}(L_p)=R_{2r+2}(L_p)$ for 
 every $r\geq 0$, thus one can think of the orders $R_{\nu}(L_p)$ as being indexed by odd positive integers. 
 These orders were studied in \cite{Pizer-orders2}, where they are called orders of level $p^{2r+1}$. 
 When $p=2$ or $L_p$ is ramified, then $R_{\nu+1}(L_p)\subsetneq R_{\nu}(L_p)$ for every $\nu\geq 1$, and 
 the order $R_{\nu}(L_p)$ depends in general on the choice of $L_p$. However, $R_2(L_p)$ is independent 
 of $L_p$, and therefore one can speak of the unique order of level $p^2$ in $D_p$ (cf. \cite{Pizer}).
\end{remark}

Assume now that $B_p = \M_2(\Q_p)$ is the split quaternion algebra over $\Q_p$. In this algebra the 
order $\M_2(\Z_p)$ is maximal and it is the only one up to conjugation by elements in $\GL_2(\Q_p)$. 
Below we introduce, for each positive integer, two different $\GL_2(\Z_p)$-conjugacy classes 
of suborders in $\M_2(\Z_p)$, which therefore define two different $\GL_2(\Q_p)$-conjugacy classes of 
orders in $\M_2(\Q_p)$. Let $\nu\geq 1$ be an integer.

\begin{itemize}

\item The subring of $\M_2(\Z_p)$ consisting of those matrices $\smallmat abcd$ in $\M_2(\Z_p)$ 
such that $p^\nu \mid c$ is commonly referred to as the {\em standard} Eichler order of level $p^{\nu}$ 
in $\M_2(\Z_p)$. An Eichler order of level $p^{\nu}$ is then any order in $\M_2(\Q_p)$ which is conjugated 
to the standard one. We shall denote any of them by $R_{\nu}^{\mathrm{Eic}}$, whenever only its 
conjugacy class is relevant in the discussion.
%

\item Let $\Q_{p^2}$ denote the unique unramified quadratic extension of $\Q_p$, and $\cO=\Z_{p^2}$ be its 
valuation ring. Then $\cO/p^{\nu}\cO$ is a finite, free, commutative $(\Z/p^{\nu}\Z)$-algebra of rank $2$ 
with unit discriminant. In particular, the choice of a basis for $\cO/p^{\nu}\cO$ gives an embedding of 
$(\cO/p^{\nu}\cO)^{\times}$ into $\GL_2(\Z/p^{\nu}\Z)$. Its image $C_{\mathrm{ns}}(p^{\nu})$ is then 
well-defined up to conjugation. The inverse image of $C_\mathrm{ns}(p^\nu)\cup\left\{\smallmat 0000\right\}$ 
under the reduction modulo $p^{\nu}$ map $\M_2(\Z_p)\rightarrow \M_2(\Z/p^{\nu}\Z)$ is an order of $\M_2(\Z_p)$, 
commonly referred to as a {\em non-split Cartan order} of level $p^{\nu}$. We shall denote any of the orders 
arising in this way simply by $R_{\nu}^{\mathrm{Car}}$, at any time that it is only the conjugacy class that 
matters in the discussion.

\end{itemize}

Now let $B/\Q$ be an indefinite quaternion algebra of discriminant $\Delta=\Delta(B)$ as before. 

\begin{defn}\label{def-ord} 
Let $N_{\mathrm{Eic}}\geq 1$ and $N_{\mathrm{Car}}\geq 1$ be such that $(N_{\mathrm{Eic}},N_{\mathrm{Car}}) = 1$ 
and $(N_{\mathrm{Eic}}\cdot N_{\mathrm{Car}},\Delta)=1$. For each prime 
$p\mid N_{\mathrm{Eic}}\cdot N_{\mathrm{Car}}$, set $\nu_p$ to be the $p$-adic valuation of 
$N_{\mathrm{Eic}}\cdot N_{\mathrm{Car}}$. For each prime $p\mid \Delta$, choose an integer 
$\nu_p \geq 1$ and a quadratic extension $L_p$ of $\Q_p$. 
An order $\mathcal R$ in $B$ is said to be {\em of type 
$T=(N_{\mathrm{Eic}};N_{\mathrm{Car}};\{(L_p,\nu_p)\}_{p\mid \Delta})$} if the following conditions 
are satisfied:  
\begin{enumerate}
\item If $p \nmid N_{\mathrm{Eic}} N_{\mathrm{Car}}\Delta$, then $\mathcal R\otimes_{\Z}\Z_p$ is a 
maximal order in $B\otimes_{\Q}\Q_p\simeq \M_2(\Q_p)$.
\item If $p \mid N_{\mathrm{Eic}}$, then $\mathcal R\otimes_{\Z}\Z_p$ is conjugate to 
$R_{\nu_p}^{\mathrm{Eic}}$ in $B\otimes_{\Q}\Q_p\simeq \M_2(\Q_p)$. 
\item If $p \mid N_{\mathrm{Car}}$, then $\mathcal R\otimes_{\Z}\Z_p$ is conjugate to 
$R_{\nu_p}^{\mathrm{Car}}$ in $B\otimes_{\Q}\Q_p\simeq \M_2(\Q_p)$. 
\item For every $p\mid \Delta$, $\mathcal R\otimes_{\Z}\Z_p \simeq R_{\nu_p}(L_p)$ in 
$B\otimes_{\Q}\Q_p\simeq D_p$.
\end{enumerate}
\end{defn}

Fix for the rest of this section an order $\mathcal R$ in $B$ of type $T=(\Ne;\Nc;\{(L_p,\nu_p)\}_{p\mid \Delta})$ 
as in Definition \ref{def-ord}. Define the \emph{level} of $\cR$ to be the integer 
$N_{\mathcal R} := \Ne\cdot \Nc^2\cdot N_{\Delta}$, where we put $N_{\Delta}:=\prod_{p\mid \Delta}p^{\nu_p}$. 
If $\nu_p=1$ for every $p\mid \Delta$, we will sometimes refer to $\mathcal R$ as a {\em Cartan--Eichler 
order} of type $(\Ne;\Nc)$ (and level $\Ne\cdot \Nc^2$).

Associated with $\mathcal R$, we have the Shimura curve $X_{\mathcal R}/\Q$ defined as in the previous 
paragraph. The Shimura curve $X_{\mathcal R}$ is projective and smooth over $\Q$, but in general it is 
not geometrically connected. The reduced norm on $\mathcal R^{\times}$ is locally surjective onto 
$\Z_{\ell}^{\times}$ at every prime $\ell \nmid \Delta$ (both Eichler and Cartan orders in indefinite 
rational quaternion algebras have class number one), but however the reduced norm on the local orders 
$R_{\nu_p}(L_p)$ is not necessarily surjective onto $\Z_p^{\times}$ when restricted to the invertible 
elements. Despite of this, it is easy to see that $[\Z_p^{\times}:\mathrm n(R_{\nu_p}(L_p)^{\times})]$ 
is either $1$ or $2$. Thus if we set
\begin{equation}\label{defC}
 \mathcal C := \{p\mid \Delta\text{ prime: } \mathrm n(R_{\nu_p}(L_p)^{\times}) \neq \Z_p^{\times} \},
\end{equation}
then the number of connected components of $X_{\mathcal R}\times_{\Q} \bar{\Q}$ is $2^{|\mathcal C|}$. 
If $\Delta$ is odd, or if $\nu_2\leq 1$ in case that $\Delta$ is even, it follows from \cite[Theorem 3.11]{HPS1} that 
\[
 \mathcal C = \{p\mid \Delta\text{ prime: } \nu_p > 1\text{, } L_p \text{ ramified} \}.
\]
The behaviour at $p=2$ is a bit more involved, but one still has a characterization of whether 
$\mathrm n(R_{\nu_2}(L_2)^{\times})$ has index $1$ or $2$ in $\Z_2^{\times}$ (see \cite[Theorem 3.11, 3) 
and 4)]{HPS1} for details). Furthermore, if $\Delta$ is odd, the connected components of 
$X_{\mathcal R}\times_{\Q} \bar{\Q}$ are defined over a polyquadratic extension: the number 
field obtained as the compositum of the quadratic extensions $\Q(\sqrt{p^*})$ for $p\in \mathcal C$, 
where $p^*=(\tfrac{-1}{p})p$.

\begin{example}\label{ex2.5}
Suppose $\Delta=pq$ with $p$ and $q$ distinct odd primes, and let $L_p$ be a quadratic ramified 
extension of $\Q_p$. Consider an order $\mathcal R$ of type $(M;1;\{(L_p,2),(L_q,1)\})$ and level 
$N=Mp^2q$. As noticed in the above remark, this order does not depend on the choice of $L_p$. The 
Shimura curve $X_{\mathcal R}/\Q$ has two geometric connected components defined over the quadratic 
field $\Q(\sqrt{p^*})$, and they are conjugated by the Galois action (in particular, they are 
isomorphic over $\Q(\sqrt{p^*})$). There is a unique Eichler order $\mathcal S$ containing $\mathcal R$, 
and the morphism of Shimura curves $X_{\mathcal R} \to X_{\mathcal S}$ induced by the inclusion 
$\hat{\mathcal R}^{\times} \subseteq \hat{\mathcal S}^{\times}$ is cyclic of degree $p+1$. Modular 
cusp forms in $S_2(\Gamma_0(N))$ which are $N/M$-new and not principal series at $p$ lift via the 
Jacquet--Langlands correspondence to quaternionic modular forms on the Shimura curve 
$X_{\mathcal R}$ (see below).
\end{example}

\section{Modular forms and the Jacquet--Langlands correspondence}

We fix throughout this section an indefinite quaternion algebra $B$ of discriminant $\Delta$ and 
an order $\cR$ of $B$ of type $T=(\Ne;\Nc;\{(L_p,\nu_p)\}_{p\mid \Delta})$ and level 
$N_{\mathcal R}=\Ne\Nc^2N_{\Delta}$. 

\subsection{Cusp forms with respect to $\mathcal R$}

We identify the Lie algebra of left invariant differential operators on 
$B_{\infty}^{\times} := (B\otimes_{\Q}\R)^{\times} \simeq \GL_2(\R)$ with $\M_2(\C)$, and 
define the differential operators 
\[
X_{\infty} = \mat{1}{\sqrt{-1}}{\sqrt{-1}}{-1}, \, \overline X_{\infty} = \mat{1}{-\sqrt{-1}}{-\sqrt{-1}}{-1}, \, 
W_{\infty} = \frac{1}{2}\mat{0}{-\sqrt{-1}}{\sqrt{-1}}{0}.
\]

\begin{defn}
 Let $k$ be an integer. A cusp form of wight $k$ with respect to $\mathcal R$ is a function 
 \[
 f: (B\otimes_{\Q}\A_{\Q})^{\times} = \hat{B}^{\times}\times\GL_2(\R) \, \, \lra \, \, \C
 \]
 satisfying the following properties:
 \begin{enumerate}
  \item if $g \in (B\otimes_{\Q}\A_{\Q})^{\times}$, then the function $\GL_2(\R) \to \C$ given by 
  $x \mapsto f(xg)$ is of $C^{\infty}$-class and satisfies $W_{\infty}f = (k/2)f$, $\overline{X}_{\infty}f=0$;
  \item for every $\gamma \in B^{\times}$ and every $u \in \hat{\mathcal R}^{\times}\times\R^{>0}$, 
  $f(u g \gamma) = f(g)$.
 \end{enumerate}
The $\C$-vector space of all cusp forms of weight $k$ with respect to $\mathcal R$ will be denoted $S_k(\mathcal R)$.
\end{defn}

The product $\prod_{p\mid \Delta} B_p^{\times}$ acts on the space $S_k(\mathcal R)$ by left translation, 
and through this action one can decompose $S_k(\mathcal R)$ into the direct sum of subspaces on which 
$\prod_{p\mid \Delta} B_p^{\times}$ acts through some admissible representation (with trivial central 
character). More precisely, suppose that  for each $p\mid \Delta$ we are given an irreducible admissible 
representation (with trivial central character) $\rho_p$ of $B_p^{\times}$ whose restriction to 
$\mathcal R_p^{\times}$ is trivial (i.e., $\mathcal R_p^{\times}\subseteq \ker(\rho_p)$). Define 
$\rho := \otimes_{p\mid \Delta}\rho_p$, regarded as a representation of $\prod_{p\mid \Delta}B_p^{\times}$. 
Since the representations $\rho_p$ are finite-dimensional, the integer 
$d_{\rho}:=\dim(\rho)=\prod_{p\mid \Delta} \dim(\rho_p)$ is well-defined.

\begin{defn}
 Let $k$ be an integer, and $\rho$ be a representation as above. A {\em cusp form of weight $k$ with 
 respect to} $(\mathcal R,\rho)$ is a function 
 \[
 f: (B\otimes_{\Q}\A_{\Q})^{\times} = \hat B^{\times}\times \GL_2(\R) \, \, \lra \, \, \C^{d_{\rho}}
 \]
 satisfying the following conditions, for every $g \in (B \otimes_{\Q}\A_{\Q})^{\times}$:
 \begin{enumerate}
  \item for every $\gamma \in B^{\times}$, $f(g\gamma) = f(g)$;
  \item for every $b \in \prod_{p\mid \Delta} B_p^{\times}$, $f(bg) = \rho(b)f(g)$;
  \item for every prime $\ell \nmid \Delta$ and $u \in \mathcal R_{\ell}^{\times}$, $f(u g) = f(g)$;
  \item the function $\GL_2(\R) \to \C^{d_{\rho}}$ given by $x \mapsto f(xg)$ is of $C^{\infty}$-class and satisfies 
  $W_{\infty}f = (k/2)f$, $\overline{X}_{\infty}f = 0$;
  \item  for every $z \in \hat{\Q}^{\times}\times\R^{\times}$, $f(gz) = f(g)$.
 \end{enumerate}
We write $S_k(\cR,\rho)$ for the $\C$-vector space of cusp forms of weight $k$ with respect to $(\mathcal R,\rho)$.
\end{defn}

The $\C$-vector spaces $S_k(\cR,\rho)$ enjoy the following multiplicity one property:  
\begin{prop}[cf. Prop. 2.14 in \cite{Hida}]\label{mult1}
 If two forms in $S_k(\cR,\rho)$ are common eigenforms of the Hecke operators $T_{\ell}$ for all 
 primes $\ell \nmid N$ with same eigenvalues, then they differ only by a constant factor.
\end{prop}

The subspace of $S_k(\mathcal R)$ on which $\prod_{p\mid \Delta} B_p^{\times}$ acts through an admissible 
representation $\rho$ as above is isomorphic to $S_k(\cR,\rho)^{d_{\rho}}$, hence one deduces that 
\begin{equation}\label{SkR:sum-rho}
S_k(\mathcal R) \simeq \bigoplus_{\rho} S_k(\cR,\rho)^{d_{\rho}},
\end{equation}
where $\rho$ ranges over the representations $\rho = \otimes_{p\mid \Delta}\rho_p$ as above, satisfying 
$\mathcal R_p^{\times}\subseteq \ker(\rho_p)$.

\begin{remark} 
The automorphic approach sketched before is related to the more classical point of view as follows. 
Let $h=h(\cR)$ denote the class number of $\mathcal R$ and choose elements $a_i \in \hat{B}^{\times}$, 
$i=1,\dots,h$, such that 
\[
\hat{B}^{\times} =\coprod_{i=1}^h \hat{\mathcal R}^{\times}a_iB^{\times}.
\]
Consider the discrete subgroups of $\SL_2(\R)$ defined by 
\[
\Gamma_i := B^{\times}_+ \cap a_i^{-1}\hat{\mathcal R}^{\times}a_i \quad \quad (i=1,\dots,h),
\]
where $B^{\times}_+$ is the subgroup of units of positive reduced norm (we may write 
$B^{\times}_+ = B^{\times}\cap \GL_2^+(\R)$ using our identification of $B\otimes_{\Q}\R$ 
with $\M_2(\R)$). If we denote by $S_k(\Gamma_i)$ the $\C$-vector space of cusp forms of 
weight $k$ with respect to the group $\Gamma_i$, then there is an isomorphism of complex 
vector spaces: 
\[
\coprod_{i=1}^h S_k(\Gamma_i)\stackrel{\simeq}{\lra} \, \, S_k(\mathcal R).
\]
\end{remark}

\subsection{Jacquet--Langlands}

The space $S_k(\cR)$ can be equipped with a standard action of Hecke operators and Atkin--Lehner 
involutions, described for example in \cite{Hida}. We have the following version of the Jacquet--Langlands 
correspondence: 

\begin{thm}[cf. Prop. 2.12 in \cite{Hida}]\label{thm:JL}
 There is a Hecke equivariant injection of $\C$-vector spaces 
 $$
 S_k(\cR,\rho) \longmono S_k(\Gamma_0(\Ne\Nc^2 N_{\rho})),
 $$
 where $N_{\rho}$ is the conductor of $\rho$.
\end{thm}

Combining Theorem \ref{thm:JL} with \eqref{SkR:sum-rho} we can embed the space $S_k(\mathcal R)$ 
into a space of classical modular cusp forms
\begin{equation}
\label{JL} 
\mathrm{JL}:
S_k(\mathcal R) \longmono  \bigoplus_{\rho} S_k(\Ne\Nc^2 N_{\rho})^{d_{\rho}}.\end{equation}
The multiplicities $d_{\rho}$ can be described explicitly: cf.\,\cite[\S 5]{Carayol}.

\begin{example}\label{ex1} 
Suppose $p\mid \Delta$ is an odd prime. The quaternion algebra $B_p = B\otimes_{\Q}\Q_p$  is equipped 
with a natural decreasing filtration $\cO_p^{\times}(i)$ defined by setting
\[
\cO_p^{\times}(0) = \cO_p^{\times} \quad \text{and} \quad \cO_p^{\times}(i) = 1 + \varpi_p^i\cO_p^{\times},
\]
where $\cO_p$ is the unique maximal order in $B_p$ and $\varpi_p$ is a local uniformizer. If $\rho_p$ 
is an admissible irreducible representation of $B_p^{\times}$, then its conductor is by definition 
$p^{n+1}$, where $n\geq 0$ is the smallest integer such that $\cO_p^{\times}(n)$ lies in the kernel 
of $\rho_p$. In particular, observe that the conductor is at least $p$. Thus, if $p^2\nmid N_{\mathcal R}$ 
then $\mathcal R_p^{\times} = \cO_p^{\times}$ is precisely the group of units in the local maximal order 
at $p$, thus the admissible irreducible representations $\rho_p$ we are concerned with all have 
conductor $p$. If $p^2\mid N_{\mathcal R}$ and $p^3\nmid N_{\mathcal R}$, we have 
$\cO_p^{\times}(1) \subseteq \mathcal R_p^{\times} \subseteq \cO_p^{\times}$, and therefore the conductor 
of the admissible irreducible representations $\rho_p$ might be either $p$ or $p^2$. For each prime 
$p\mid \Delta$, the dimension of $\rho_p$ is determined by its {\em minimal conductor}, which by 
definition is the smallest conductor of the representations $\rho_p\otimes \chi$, as $\chi$ ranges 
over the characters of $\Q_p^{\times}$.  Let $N^{\mathrm{min}}_{\rho_p}$ be the minimal conductor 
of $\rho_p$. By \cite[\S 5]{Carayol}, if $N^{\mathrm{min}}_{\rho_p}=p^a$ with $a\in \{1,2\}$, then 
$d_{\rho_p} = a$.
\end{example}

The above arguments give us a Hecke equivariant inclusion of $S_k(\mathcal R)$ into a direct sum of 
spaces of classical modular cusp forms. In order to circumvent the problem of explicitly determining 
the multiplicities $d_{\rho}$, we use Proposition \ref{propJL} below, which benefits from an explicit 
version of Eichler trace formula due to Hijikata, Pizer and Shemanske. 

For the reader's convenience, we recall the classification of Jacquet--Langlands lifts given in 
\cite{HPS2}, and from now on we focus on the weight $2$ case. So let $f\in S_2(\Gamma_0(N_f))$ be a 
weight $2$ modular cusp form, and assume that $N_f = p^sM$ for some prime $p$ and integers $s, M \geq 1$, 
with $p\nmid M$. Suppose $\phi$ is a Jacquet--Langlands lift of $f$ which is realized on the definite 
quaternion algebra $B^{(p)}$ of discriminant $p$. We want to determine the level of $\phi$, by which 
we mean the local $p$-type of the order $\mathcal R$ of $B^{(p)}$ used to define its level structure. 
Such local order is of the form $R_n(L)$, for some positive integer $n$ and quadratic extension $L/\Q_p$, 
and it is determined as follows:


\begin{enumerate}

\item if $p$ is odd: 

\begin{enumerate}

\item $s$ odd: $L$ is unramifield and $n=s$ (\cite[Theorem 8.5]{HPS2}).  

\item $s$ even: $L$ is ramified (any of the two ramified extensions) and $n=s$ (\cite[Proposition 8.8 Case D]{HPS2}). 

\end{enumerate} 

\item if $p=2$: 

\begin{enumerate}

\item $s=1$: $L$ is the unramified quadratic extension of $\Q_2$ and $n=1$ (\cite[Proposition 8.8 Case C]{HPS2}). 

\item $s$ odd, $s\geq 3$: $L$ is unramifield and $n=s$ (\cite[Theorem 8.5]{HPS2}). 

\item $s=2$: $L=\Q_2(\sqrt{3})$ or $L=\Q_2(\sqrt{7})$ and $n=2$ \cite[Proposition 8.8 Case F Eq. (8.17)]{HPS2}. 

\item $s$ even, $s\geq4$: In this case, $f$ is a twist by a non-trivial character of conductor $2^{s/2}$ of one of the 
previous cases (\cite[Theorem 3.9]{HPS2}). 


\end{enumerate} 
\end{enumerate}

\begin{prop}\label{propJL}
Let $f\in S_2(\Gamma_0(N_f))$ be a newform and fix a set $\Sigma$ of even cardinality consisting of primes 
$\ell\mid N_f$ such that the local admissible representation $\pi_{f,\ell}$ of $\GL_2(\Q_\ell)$ attached to 
$f$ is square-integrable. Let $B/\Q$ be the indefinite quaternion algebra of discriminant 
$\Delta=\prod_{\ell \in \Sigma} \ell$, and $\Ne, \Nc$ be positive integers such that $(\Ne,\Nc)=1$, 
$\Ne\Nc^2 \mid N_f$ and $(\Ne\Nc, \Delta) = 1$. 

Then there exists an order $\cR_\mathrm{min}\subset B$ of type 
$T_\mathrm{min} = (\Ne;\Nc; \{(L_p,\nu_p)\}_{p\mid\Delta})$ such that $f$ lifts to a quaternionic modular 
form on $S_2(\mathcal R_\mathrm{min})$ having the same Hecke eigenvalues for the Hecke operators $T_{\ell}$ 
at primes $\ell \nmid N_f$. Further, for each prime $p\mid \Delta$ the data $(L_p,\nu_p)$ depends only on 
$\val_p(N_f)$.
\end{prop}

The subscript `min' in $\cR_\mathrm{min}$ refers to the minimal level for primes dividing $\Delta$, 
determined by the classification explained above. We note that if $\nu_p$ is odd then $L_p$ is unramified, 
and if $\nu_p$ is even then $L_p$ is ramified. 

\begin{proof}
Since $\pi$ is square integrable at all primes in $\Sigma$, \cite[Theorem 10.2]{Gel} implies the existence 
of an automorphic form $\pi'$ on the algebraic group of invertible elements of the indefinite quaternion 
algebra $B$ as in the statement such that $\pi_{\ell}' \simeq \pi_{\ell}$ for all primes $\ell$. To specify 
the order $\cR$ we need to describe $\pi_{\ell}'$ at every prime $\ell$. For primes $\ell \nmid \Nc\Delta$ the 
assertion is obvious. Fix a prime $p \mid \Delta$ and let $B^{(p)}$ be the definite quaternion algebra of 
discriminant $p$. Then Eichler's trace formula in \cite{HPS2} shows the existence of an automorphic form 
$\pi^{(p)}$ for $B^{(p)}$ attached to a specific order $\cR^{(p)}$ of type $(L,\nu(L))$ depending only on 
$\val_p(N_f)$ with $\pi'_p \simeq \pi^{(p)}_p$ (we have sketched above the recipe for $L$ and $\nu(L)$). 
Finally, for primes dividing $\Nc$ a similar argument works using this time the trace formula in 
\cite[Sec. 6]{chen2} (the proof in \cite{chen2} only works for $p\neq 2$, but one can check that
it can be extended to the case $p=2$). 
\end{proof}

\begin{remark}\label{remarkJL}
 Let $N_{\mathrm{min}}$ be the level of the order $\mathcal R_{\min}$ in the proposition. Then observe that 
 $N_{\mathrm{min}}$ divides $N_f$. Even more, we have $\val_p(N_{\mathrm{min}}) = \val_p(N_f)$ for every 
 odd prime $p$. And in case that $N_f$ is even, then $\val_2(N_{\mathrm{min}}) < \val_2(N_f)$ implies that 
 $2 \mid \Delta$ and $\val_2(N_f)$ is even and at least $4$. 
\end{remark}

\begin{example} \label{ex2}
It follows from \cite{Pizer} that, given a primitive (in the sense of \cite[Definition 8.6]{Pizer}) new 
cuspidal eigenform $f$ of level $\Gamma_0(p^2M)$ as in the proposition with $p \nmid M$ an odd prime, the 
subspace of new forms in $S_2(\mathcal R)$ having the same system of Hecke eigenvalues as $f$ (at primes 
outside $N_f = p^2M$) is two-dimensional. So we have a ``multiplicity 2 phenomenon'' as expected from the 
Example \ref{ex1} and Hecke-equivariant monomorphism $\mathrm{JL}$ in \eqref{JL}.
\end{example}

\subsection{Modular parametrizations}

Let $J_{\mathcal R}/\Q$ denote the Jacobian variety of $X_{\mathcal R}$. It is a (principally polarized) abelian 
variety defined over $\Q$, of dimension equal to the genus of $X_{\mathcal R}$. Since $X_{\mathcal R}$ is not in 
general geometrically connected, it follows that $J_{\mathcal R}$ might be not absolutely simple. Recall the following: 

\begin{defn}\label{def-modular}
An abelian variety $A/\Q$ is said to be \emph{modular} if there exists a normalized
newform $f = \sum_{n\geq 1}a_nq^n$ of weight $2$ and level $\Gamma_0(N_f)$ for some $N_f\geq 1$ such that 
\[
L(A,s)=\prod_{\sigma: F \hookrightarrow \bar\Q}L(f^\sigma,s),
\] 
where $F$ stands for the number field generated by the Fourier coefficients of $f$, $\sigma$ ranges over 
the embeddings of $F$ into an algebraic closure of $\Q$ and $f^\sigma=\sum_{n\geq 1}\sigma(a_n)q^n$. 
\end{defn}

\begin{prop}
Suppose that $A/\Q$ is a modular abelian variety associated with a modular form $f=\mathrm{JL}(\phi)$ 
for some $\phi\in S_2(\cR)$. Let $\mathbb I_\phi \subseteq \T$ be the kernel of the ring homomorphism $\T \to \Z$ 
determined by the system of Hecke eigenvalues of $\phi$. Then the quotient abelian variety 
$A_\phi := J_{\mathcal R}/\mathbb I_\phi J_{\mathcal R}$ is isogenous to $A^r$ for some $r\geq 1$. 
\end{prop}
\begin{proof}
Let $\ell\nmid N_f$ be a prime and $\varrho: G_{\Q} \to \Aut(\Ta_{\ell}(A)\otimes \Q_{\ell})$ be the $2$-dimensional 
$\ell$-adic Galois representation arising from the natural action of $G_{\Q}$ on the $\ell$-adic Tate module 
$\Ta_{\ell}(A)$ of $A$. Similarly, let $\theta: G_{\Q} \to \Aut(\Ta_{\ell}(A_\phi)\otimes \Q_{\ell})$ be the 
$\ell$-adic Galois representation attached to $A_\phi$. The Eichler--Shimura relations (proved in the required 
generality in \cite{NekovarES}) imply that $\theta(\sigma)$ is annihilated by the characteristic polynomial of 
$\varrho(\sigma)$ for every $\sigma \in G_{\Q}$. Then the Boston--Lenstra--Ribet Theorem \cite{BLR} implies that 
$\Ta_{\ell}(A_\phi)\otimes \Q_{\ell}$ is isogenous to a direct sum of $r$ copies of $\Ta_{\ell}(A)\otimes \Q_{\ell}$ 
for some $r\geq 1$. Finally, Faltings' Isogeny Theorem implies that $A_\phi$ is isogenous to $r$ copies of $A$.
\end{proof}

\begin{example}
Suppose that $\cR$ is of type $(M;1;\{(L_p,2),(L_q,1)\}$ and level $N:=N_{\mathcal R} = Mp^2q$, with $p$ and $q$ 
distinct odd primes as in Example \ref{ex2.5}. Set $F = \Q(\sqrt{p^*})$. Then 
\[
J_{\mathcal R} \times_{\Q} F \sim J_{\mathcal R,1} \times J_{\mathcal R,2},
\]
where $J_{\mathcal R,i}/F$ is the Jacobian variety of $X_{\mathcal R,i}$. Let $\mf S_2(\Gamma_0(N))$ be the 
subspace of $S_2(\Gamma_0(N))$ consisting of primitive newforms. From Example \ref{ex2} we know that there is 
a 2-to-1 Hecke-equivariant morphism of $\C$-vector spaces 
\[
\mf S_2(\mathcal R) \, \, \lra \, \, \mf S_2(\Gamma_0(N)),
\]
where $\mf S_2(\mathcal R)$ is the subspace of modular forms $\phi \in S_2(\mathcal R)$ such that 
$\mathrm{JL}(\phi) \in \mf S_2(\Gamma_0(N))$. By a slight abuse of notation we continue to denote by 
$\mathrm{JL}$ this morphism.

Fix $f\in \mf S_2(\Gamma_0(N))$ and assume that the Fourier coefficients of $f$ belongs to $\Z$. Then the 
abelian variety associated with $f$ is an elliptic curve $E$ of conductor $Mp^2q$. Let $\phi\in \mf S_2(\cR)$ 
be such that $\mathrm{JL}(\phi)=f$ (we have two linearly independent possible choices). The space 
$S_2(\mathcal R)$ of weight $2$ modular forms for $\mathcal R$ is identified with $H^0(X_{\mathcal R},\Omega^1)$, 
which in turn is identified with the tangent space at the identity $T_0(J_{\mathcal R})$ of $J_{\mathcal R}$. 
The subspace $\mf S_2(\mathcal R)$ corresponds then to a subspace of $H^0(X_{\mathcal R},\Omega^1)$, and 
hence to the tangent space $T_0(\mf J_{\mathcal R})$ of an abelian subvariety $\mf J_{\mathcal R}$ of 
$J_{\mathcal R}$. The space of modular forms $\mf S_2(\mathcal R)$ has rank $2$ over the Hecke ring $\T$, and 
from this it follows that $T_0(\mf J_{\mathcal R})/T_0(\mathbb I_\phi\mf J_{\mathcal R})$ has dimension $2$ 
over $\Q$. Hence $A_\phi$ is $2$-dimensional, and therefore $A_\phi\sim E\times E$.
\end{example}

\section{Heegner points}

\subsection{Optimal embeddings}

As in previous sections, $B$ denotes an indefinite rational quaternion algebra of discriminant 
$\Delta=\Delta(B)$. We fix an order $\cR$ and a quadratic field $K$. For each positive integer 
$c$ write $R_c$ for the (unique) order of conductor $c$ in $K$, $R_1$ being the full ring of 
integers of $K$. 

\begin{defn}
 Let $c$ be a positive integer. An embedding from $K$ to $B$, i.e. a $\Q$-algebra homomorphism 
 $f: K \, \to \, B$, is said to be {\em optimal with respect to $\mathcal R/R_c$} if the equality
 $$
 f(K) \cap \mathcal R = R_c
 $$
 holds. Since $f$ is determined by its restriction to $R_c$, we also speak of optimal embeddings of $R_c$ 
 into $\mathcal R$.
\end{defn}

Two optimal embeddings $f, f'$ of $R_c$ into $\mathcal R$ will be considered to be equivalent if they are 
conjugate one to each other by an element in $\mathcal R^{\times}$. The set of $\mathcal R^{\times}$-conjugacy 
classes of optimal embeddings of $R_c$ into $\mathcal R$ will be denoted $\Emb^{\op}(R_c,\mathcal R)$. We 
are interested in computing the integer 
\[
v(R_c,\mathcal R) = |\Emb^{\op}(R_c,\mathcal R)|,
\]
and in particular in knowing whether the set $\Emb^{\op}(R_c,\mathcal R)$ is empty or not.

Suppose now that $\cR$ is of type $T=(\Ne;\Nc;\{(L_p,\nu_p)\}_{p\mid \Delta})$ and level 
$N_{\mathcal R}=\Ne\Nc^2N_{\Delta}$. Recall that the class number $h(\mathcal R)$ of the 
order $\mathcal R$ is $2^{|\mathcal C|}$, where $\mathcal C$ is the set introduced in \eqref{defC}.
Although the class number of $\mathcal R$ is therefore not trivial in general, the lemma below 
asserts that the {\em type number} of $\cR$ is always trivial, which amounts to saying that all 
orders in $B$ of the same type $T$ are conjugate one to each other.

\begin{lemma}
 The type number of orders of a fixed type $T$ is $1$.
\end{lemma}
\begin{proof}
Fix a type $T=(\Ne;\Nc;\{(L_p,\nu_p)\}_{p\mid \Delta})$ as in Definition \ref{def-ord}, and let $\mathcal R$ and 
$\mathcal R'$ be two orders of type $T$ in $B$. First of all, notice that $\mathcal R$ (resp. $\mathcal R'$) is 
a suborder of a unique Cartan--Eichler order $\mathcal S$ (resp. $\mathcal S'$) of level $\Ne\Nc^2$. Namely, the 
order which is locally equal to $\mathcal R$ (resp. $\mathcal R'$) at every prime $p\nmid \Delta$ and locally maximal 
at primes $p\mid\Delta$, hence of type $(\Ne;\Nc;\{(L_p,1)\}_{p\mid \Delta})$. Conversely, it is clear by construction 
that $\mathcal R$ (resp. $\mathcal R'$) is the unique suborder of type $T$ of the Cartan--Eichler order $\mathcal S$ 
(resp. $\mathcal S'$). The lemma now follows from the fact that the type number of Cartan--Eichler orders in $B$ 
is $1$, so that $\mathcal S$ and $\mathcal S'$ are conjugate. By the above observation, this immediately implies 
that $\mathcal R$ and $\mathcal R'$ are conjugate as well.
\end{proof}

By virtue of the above lemma, the number $v(R_c,\mathcal R)$ can be expressed essentially as a product of 
local contributions that can be explicitly computed. Indeed, proceeding as in the proof of the `trace formula' 
in \cite[Ch. III, 5.C]{Vi} (cf. especially Theorems 5.11 and 5.11 bis, or \cite{Brz1}) for Eichler orders, we 
have that
\begin{equation}\label{opt-emb}
 v(R_c,\mathcal R) = \frac{h(R_c)}{h(\mathcal R)}\prod_{\ell} v_{\ell}(R_c,\mathcal R),
\end{equation}
where $h(R_c)$ (resp. $h(\mathcal R)$) is the class number of the quadratic order $R_c$ (resp. of $\mathcal R$), 
the product ranges over all rational primes and, for each $\ell$, $v_{\ell}(R_c,\mathcal R)$ denotes the number 
of local optimal embeddings of $R_c\otimes_{\Z}\Z_{\ell}$ into $\mathcal R\otimes_{\Z}\Z_{\ell}$ modulo conjugation 
by $(\mathcal R\otimes_{\Z}\Z_{\ell})^{\times}$. These local contributions are $1$ for every prime $\ell \nmid N$. 
The number of local optimal embeddings is determined in \S \ref{opt-emb-sec} below. Here we give the following: 

\begin{example}
Assume that $\Nc=1$, $\Delta$ is odd and $\nu_p\leq 2$ for all $p\mid \Delta$. Then
\[v(R_c,\mathcal R) = \frac{h(R_c)}{h(\mathcal R)} 
\prod_{\ell\mid \Ne}\left( 1+\left\lbrace \frac{R_c}{\ell}\right\rbrace\right)
\prod_{\substack{q \mid \Delta\\ \nu(q)=1}}\left( 1 - \left\lbrace \frac{R_c}{\ell}\right\rbrace\right)
\prod_{\substack{p \mid \Delta\\ \nu_p=2}}v_p(R_c,\mathcal R),
\]
where for primes $p\mid \Delta$ with $\nu_p = 2$,
\[
 v_p(R_c,\mathcal R) = 
 \begin{cases}
  2 & \mbox{if } p\mid\mid c \text{ and } p \text{ is inert in } K,\\
  p+1 & \mbox{if } p\nmid c \text{ and } p \text{ ramifies in } K,\\
  0 & \mbox{otherwise.}
 \end{cases}
\]
Here $\left\{\frac{R}{\ell}\right\}$ denotes the usual Eichler symbol attached to a quadratic order $R$ and a prime number $\ell$.
\end{example}

\subsection{Local optimal embeddings}\label{opt-emb-sec}
For the reader's convenience, we reproduce in this subsection the criteria for the existence of local optimal 
embeddings of orders in quadratic fields into quaternion orders in the Eichler, Cartan and division cases.

\subsubsection{Eichler case} \label{appA1}

Let $p$ be a prime, $K/\Q_p$ be a quadratic separable algebra, and $\cO_m \subseteq K$ be the order in $K$ of 
conductor $p^m$. Let also $\M_2(\Q_p)$ be the split quaternion algebra over $\Q_p$ and $R_n^{\mathrm{Eic}}$ be 
the standard Eichler order of level $p^n$ in $\M_2(\Q_p)$. Write $h(m,n)$ for the number of (equivalence classes 
of) optimal embeddings of $\cO_m$ into $R_n^{\mathrm{Eic}}$.

Suppose first that $K=\Q_p\oplus\Q_p$ is the split quadratic $\Q_p$-algebra. Then $m$ is the smallest positive 
integer such that $\cO_m/p^m\cO_m \simeq \Z/p^m\Z$. In this case, the embedding $(a,b)\mapsto \smallmat a0{p^{n-m}(a-b)}b$ 
from $K$ into $\M_2(\Q_p)$ defines an optimal embedding from $\cO_m$ into $R_n^{\mathrm{Eic}}$. For later reference, 
we state the following lemma.

\begin{lemma}\label{LemmaC0} 
If $K$ is the split quadratic $\Q_p$-algebra, then $\cO_m$ can be optimally embedded in $R_n^{\mathrm{Eic}}$ for 
every $m\geq 0$. That is, $h(m,n)\neq 0$ for every $m\geq 0$.
\end{lemma} 

Next we assume that $K/\Q_p$ is a quadratic field extension with valuation ring $\cO$, and again for each $m\geq 1$ 
let $\cO_m$ be the order of conductor $p^m$ in $K$. Recall that the Eichler symbol is defined as follows: 
\[
\left\{\frac{\cO_m}{p}\right\}=\begin{cases} 
-1 & \text{if $m=0$ and $K/\Q_p$ is unramified};\\
0 & \text{if $m=0$ and $K/\Q_p$ is ramified};\\
1 & \text{if $m\geq 1.$}
\end{cases}
\]
It is well known (\cite{Hijikata}, \cite{Vi}) that if $n = 0$ then $h(m,n) = 1$, and for $n=1$ one has 
$h(m,n)= 1 +\left\{\frac{\cO_m}{p}\right\}$. Thus, in particular, every quadratic order $\cO_m$ can be optimally 
embedded in the maximal Eichler order unless $m=0$ and $K/\Q_p$ is unramified, the only case when $h(m,1)=0$. 
More generally (see \cite[Corollary 1.6]{brz2}): 

\begin{lemma}\label{lemmaC1}
\begin{enumerate} 
\item If $K/\Q_p$ is unramified, then $h(m,n)\neq 0$ if and only if $m\geq n/2$.
\item If $K/\Q_p$ is ramified, then $h(m,n)\neq 0$ if and only if $m\geq (n-1)/2$.
\end{enumerate}
\end{lemma}

\subsubsection{Cartan Case} 

Let $p$ be a prime, $K=\Q_{p^2}$ be the unramified quadratic extension of $\Q_p$ and $\cO=\Z_{p^2}$ be 
its valuation ring. As above, for $m\geq 1$ write $\cO_m$ for the order of conductor $p^m$ in $K$. From 
the very definition of non-split Cartan orders, we have the following lemma, which we state for later reference:

\begin{lemma}\label{LemmaC2}
 Let $R_n^{\mathrm{Car}}$ be a non-split Cartan order of level $p^n$ in $\M_2(\Q_p)$. Then $\cO$ can be optimally 
 embedded in $R_n^{\mathrm{Car}}$. For $m > 1$, the order $\cO_m$ does embed in $R_n^{\mathrm{Car}}$, but 
 {\em not optimally}.
\end{lemma}

%
%

\subsubsection{Division case}\label{appA} 

References: \cite{HPS1}. Let $p$ be a prime, and $D_p$ be the unique division quaternion algebra over $\Q_p$. As above, 
write $R_n(L)$ for the local order in $D_p$ associated to the choice of an integer $n\geq 1$ and a quadratic extension 
$L/\Q_p$. Let $K/\Q_p$ be a quadratic field extension, and $\cO_m$ denote the order of conductor $p^m$ in $K$ as before. 
Recall that $h(m,n,L)$ denotes the number of equivalence classes of optimal embeddings of $\cO_m$ into $R_n(L)$. 

It might be useful first to recall the notation used in \cite{HPS1} for the symbols $t(L)$ and $\mu(L,L')$. 
For any quadratic field extension $L/\Q_p$: 
\begin{itemize}
\item $t(L)=-1$ means $L$ unramified; 
\item $t(L)=0$ means $L$ ramified and $p\neq 2$; 
\item $t(L)=1$ means $p=2$ and $L=\Q_p(\sqrt{3})$ or $L=\Q_p(\sqrt{7})$; 
\item $t(L)=2$ means $p=2$ and $L=\Q_p(\sqrt{2})$, $L=\Q_p(\sqrt{6})$, $L=\Q_p(\sqrt{10})$ or $L=\Q_p(\sqrt{14})$.
\end{itemize}
And for any pair of quadratic field extensions $(L,L')$ of $\Q_p$ having discriminants 
$\Delta(L)$ and $\Delta(L')$ we have: 
\begin{itemize}
\item $\mu(L,L')=\mu(L',L)$ (Theorem 3.10 A (iii) of \cite{HPS1}); 
\item If $\Delta(L)=\Delta(L')$ (which is the case if $L\simeq L'$) then $\mu(L,L')=\infty$; 
\item If $t(L)=-1$ and $\Delta(L)\neq \Delta(L')$ then $\mu(L,L')=1$; 
\item If $t(L)=0$, $t(L')=0$ and $\Delta(L)\neq \Delta(L')$ then $\mu(L,L')=2$; 
\item If $t(L)=1$, $t(L')=1$ and $\Delta(L)\neq \Delta(L')$ then $\mu(L,L')=3$; 
\item If $t(L)=1$, $t(L')=2$ then $\mu(L,L')=3$; 
\item If $t(L)=2$, $t(L')=2$ and $\Delta(L)\neq \Delta(L')$ then $\mu(L,L')=5$. 
\end{itemize}

The criteria for the existence of optimal embeddings then reads as follows: 

\begin{enumerate}

\item $p$ odd: 

\begin{enumerate}

\item $n=2\varrho+1$ odd, $K$ unramified,  $L$ unramified: $h(m,n,L)\neq 0$ if and only if $m\leq\varrho$. 
In particular, if $R_n(L)$ is maximal and $\cO_m$ is not maximal (\emph{i.e.} $m>0$ and $n=0$) then $h(m,n,L)=0$. 

\item $n=2\varrho+1$ odd, $K$ ramified, $L$ unramified: $h(m,n,L)\neq 0$ if and only if $m=\varrho$. 
In particular, if $R_n(L)$ is maximal and $\cO_m$ is not maximal (\emph{i.e.} $m>0$ and $n=0$) then $h(m,n,L)=0$. 

\item $n=2\varrho$ even, $K$ unramified, $L$ ramified: $h(m,n,L)\neq 0$ if and only if $m=\varrho$. 

\item $n=2\varrho$ even, $K$ ramified, $L$ ramified and $K\not\simeq L$: $h(m,n,L)\neq 0$ if and only 
if $m=\varrho-1$. 

\item $n=2\varrho$ even, $K$ ramified, $L$ ramified and $K\simeq L$: $h(m,n,L)\neq 0$ if and only if 
$m\leq\varrho-1$. 
\end{enumerate}

\item $p=2$: 

\begin{enumerate}

\item $n=1$, $K$ ramified or unramified, $L$ unramified: $h(m,n,L)\neq 0$ if and only if $m=0$. 

\item $n=2\varrho$, $K$ unramified, $L=\Q_2(\sqrt{3})$ or $L=\Q_2(\sqrt{7})$; this is the case 
of $t(L)=1$, $t(K)=-1$ and therefore $\mu(L,K)=1$: $h(m,n,L)\neq 0$ if and only if $m=\rho$. 

\item $n=2\varrho$, $K=\Q_2(\sqrt{3})$ or $K=\Q_2(\sqrt{7})$, $L=\Q_2(\sqrt{3})$ or $L=\Q_2(\sqrt{7})$ 
and $K\not\simeq L$; this is the case of $t(L)=1$, $t(K)=1$ and $\Delta(L)\neq \Delta(K)$ and therefore 
$\mu(L,K)=3$: $h(m,n,L)\neq 0$ if and only if $m=\rho-1$. 

\item $n=2\varrho$, $K=\Q_2(\sqrt{3})$ or $K=\Q_2(\sqrt{7})$ and $K\simeq L$; this is the case of 
$t(L)=1$, $t(K)=1$ and $\Delta(L)= \Delta(K)$ and therefore $\mu(L,K)=\infty$: $h(m,n,L)\neq 0$ 
if and only if $m\leq \rho-1$. 

\item $n=2\varrho$, $K=\Q_p(\sqrt{2})$, $K=\Q_p(\sqrt{6})$, $K=\Q_p(\sqrt{10})$ or $K=\Q_p(\sqrt{14})$, 
$L=\Q_2(\sqrt{3})$ or $L=\Q_2(\sqrt{7})$; this is the case of $t(L)=1$, $t(K)=2$ and therefore $\mu(L,K)=3$: 
$h(m,n,L)\neq 0$ if and only if $m=\rho-1$. 

\item $n=2\varrho+1$ odd, $n\geq 3$, $K$ unramified and $L$ unramified: $h(m,n,L)\neq 0$ if and only 
if $m\leq \varrho$. 

\item $n=2\varrho+1$ odd, $n\geq 3$, $K$ ramified and $L$ unramified: $h(m,n,L)\neq 0$ if and only 
if $m=\varrho$. 

\item $n=2\varrho$, $K$ unramified and $L$ ramified. Then $t(L)=1$ or $2$ and $\mu(L,K)=1$: $h(m,n,L)\neq 0$ if 
and only if $m=\varrho$. 


\item $n=2\varrho$, $K=\Q_2(\sqrt{3})$ or $K=\Q_p(\sqrt{7})$, $L=\Q_2(\sqrt 2)$,  $L=\Q_2(\sqrt 6)$,
$L=\Q_2(\sqrt{10})$ or $L=\Q_2(\sqrt{14})$. Then $t(L)=2$, $t(K)=1$, $\mu(L,K)=3$: $h(m,n,L)\neq 0$ if 
and only if $m= \varrho-1$.

\item $n=2\varrho$, $K=\Q_2(\sqrt 2)$,  $K=\Q_2(\sqrt 6)$, $K=\Q_2(\sqrt{10})$ or $K=\Q_2(\sqrt{14})$, 
$L=\Q_2(\sqrt 2)$,  $L=\Q_2(\sqrt 6)$, $L=\Q_2(\sqrt{10})$ or $L=\Q_2(\sqrt{14})$ and $K\not\simeq L$. 
Then $t(L)=2$, $t(K)=2$, $\mu(L,K)=5$: $h(m,n,L)\neq 0$ if and only if $m = \varrho-1$ or $m=\varrho-2$.

\item $n=2\varrho$, $K=\Q_2(\sqrt 2)$,  $K=\Q_2(\sqrt 6)$, $K=\Q_2(\sqrt{10})$ or $K=\Q_2(\sqrt{14})$, 
and $K\simeq L$. Then $t(L)=2$, $t(K)=2$, $\mu(L,K)=\infty$: $h(m,n,L)\neq 0$ if and only if 
$m\leq \varrho-1$.
\end{enumerate}
\end{enumerate}

\subsection{Heegner points}

Let $U$ be any open compact subgroup of $\hat B^\times$, and assume that $K$ is an imaginary quadratic 
field. There is a natural map 
\[
\hat B^{\times}\times \Hom(K,B) \, \, \lra \, \, 
\left(U \backslash \hat{B}^{\times} \times \Hom(\C,\M_2(\R))\right)/B^{\times} = X_U(\C)
\]
obtained by extending scalars (i.e., tensoring with $\R$). Notice that the left-hand side can certainly 
be the empty set, as $\Hom(K,B)$ is empty if $K$ does not embed into $B$. We shall assume that this is 
not the case in the discussion below. If $(g,f) \in \hat B^{\times}\times \Hom(K,B)$, write $[g,f]$ for 
its image in $X_{\mathcal R}(\C)$. Points in the image of this map are called \emph{Heegner points}; the set 
of such Heegner points is denoted $\mathrm{Heeg}(U,K)$. 

For each positive integer $c$, continue to denote by $R_c$ the order of conductor $c$ in $K$ and let 
$\mathcal R$ be an order of $B$. 

\begin{defn}
 A point $x\in X_{\mathcal R}$ is called a {\em Heegner point of conductor $c$ associated to $K$}
 if $x = [g,f]$ for some pair $(g,f) \in \hat B^{\times}\times \Hom(K,B)$ such that 
 $$
 f(K) \cap g^{-1}\hat{\mathcal R}g = f(R_c).
 $$
 This last condition means that $f$ is an optimal embedding of $R_c$ into the order 
 $g^{-1}\hat{\mathcal R}g \cap B$. We shall denote by $\mathrm{Heeg}(\mathcal R,K,c)$ the set of 
 Heegner points of conductor $c$ associated to $K$ in $X_{\mathcal R}$. 
\end{defn}

Recall that the set of geometrically connected components of the Shimura curve $X_{\mathcal R}$ 
is in bijection with $\hat{\mathcal R}^{\times}\backslash \hat{B}^{\times}/B^{\times}$, and hence 
with the class group $\Pic(\mathcal R)$ of $\mathcal R$. In particular, the number of geometric 
connected components coincides with the class number $h(\mathcal R)$. Fix representatives $I_j$ 
for the distinct $h(\mathcal R)$ ideal classes in $\Pic(\mathcal R)$, and let $a_j \in \hat{B}^{\times}$ 
be the corresponding representatives in $\hat{\mathcal R}^{\times}\backslash \hat{B}^{\times}/B^{\times}$. 
It is then clear that every Heegner point in $\mathrm{Heeg}(\mathcal R,K,c)$ can be represented 
by a pair of the form $(a_j,f)$, for a unique $j \in \{1,\dots,h(\mathcal R)\}$ and some optimal 
embedding $f$ from $R_c$ into the order $a_j^{-1}\hat{\mathcal R}a_j \cap B$. Further, two pairs $(a_j,f)$ and $(a_j,g)$ 
represent the same Heegner point if and only if the embeddings $f$ and $g$ are 
$\mathcal R^{\times}$-conjugate. Therefore, we have the following identity relating Heegner points 
on $X_{\mathcal R}$ attached to $R_c$ and optimal embeddings of $R_c$ into $\mathcal R$:
\[
  |\mathrm{Heeg}(\mathcal R,K,c)| = h(\mathcal R)|\Emb^{\op}(R_c,\mathcal R)| = h(\mathcal R)v(R_c,\mathcal R),
\]
thus applying \eqref{opt-emb} we find:

\begin{prop}
 The number of Heegner points on $X_{\mathcal R}$ attached to $R_c$ is $h(R_c)\prod_{\ell} v_{\ell}(R_c,\mathcal R)$.
\end{prop}

\subsection{Galois action and fields of rationality}

Keep the same notations as above, and assume that $R_c$ embeds optimally in $\mathcal R$, 
so that Heegner points with respect to $R_c$ do exist on $X_{\mathcal R}$. The reciprocity law, 
cf. \cite[3.9]{DeligneTravauxShimura}, \cite[II.5.1]{MilneModels} (with a sign corrected 
\cite[1.10]{MilneShimuramodp}), asserts that $\mathrm{CM}(\mathcal R,K,c) \subseteq X_{\mathcal R}(K^{ab})$, 
where as usual $K^{ab}$ denotes the maximal abelian extension of $K$, and further that the 
Galois action of $\Gal(K^{ab}/K)$ on $\mathrm{CM}(\mathcal R,K,c)$ is described by 
\begin{equation}\label{GaloisCM}
\rec_K(a)[g,f] = [\hat f(a)g,f], \qquad (a \in \hat K^{\times}).
\end{equation}
Here, $\rec_K: \hat K^{\times} \to \Gal(K^{ab}/K)$ is the reciprocity map from class field 
theory. Then, for an arbitrary $a\in \hat K^{\times}$ and every Heegner point $[g,f]$ we have
$$
\rec_K(a)[g,f] = [g,f] \iff \text{ there exist } b\in B^{\times}, u\in \hat{\mathcal R}^{\times} \text{ such that } (\hat f(a)g,f) = (ugb,b^{-1}f b).
$$
It is easy to show that if $f: K \to B$ is an embedding and $b\in B^{\times}$, then the 
equality $f=b^{-1}fb$ holds if and only if $b=f(\lambda)$ for some $\lambda \in K^{\times}$. 
Thus we deduce that 
\begin{eqnarray*}
\rec_K(a)[g,f] = [g,f] & \iff & \text{there exist } \lambda \in K^{\times}, u\in \hat{\mathcal R}^{\times} \text{ such that } 
\hat f(a) = g^{-1}ugf(\lambda) \\
& \iff & a \in \hat f^{-1}(g^{-1}\hat{\mathcal R}^{\times}g)K^{\times} = \hat{R}_c^{\times}K^{\times}.
\end{eqnarray*}

By class field theory, $\rec_K$ induces an isomorphism 
$$
\hat K^{\times}/\hat{R}_c^{\times}K^{\times} = \Pic(R_c) \stackrel{\simeq}{\lra} \Gal(H_c/K),
$$
where $H_c$ is the ring class field of conductor $c$. Hence we have proved:

\begin{prop}
 With notations as above, $\mathrm{Heeg}(\mathcal R,K,c) \subseteq  X_{\mathcal R}(H_c)$, and 
 the action of $\Gal(H_c/K)$ on the set of Heegner points $\mathrm{Heeg}(\mathcal R,K,c)$ is 
 described by the rule in \eqref{GaloisCM}.
\end{prop}

\section{Applications}

\subsection{Gross--Zagier formula} 

We briefly review the general form of Gross--Zagier formula in \cite{YZZ} for modular abelian 
varieties. Let $B/\Q$ be an indefinite quaternion algebra of discriminant $\Delta$. If $U_1\subseteq U_2$ 
are open compact subgroups of $\hat B^\times$, then we have a canonical projection map 
$\pi_{U_1,U_2}:X_{U_1}\twoheadrightarrow X_{U_2}$, and one may consider the projective limit 
\[
X=\invlim_U X_U,
\]
and let $J:=\Jac(X)$ denote the Jacobian variety of $X$.

\begin{defn}
A simple abelian variety $A/\Q$ is said to be \emph{uniformized by $X$} if there exists
a surjective morphism $J \twoheadrightarrow A$ defined over $\Q$.
\end{defn}

Let $A/\Q$ be a simple abelian variety uniformized by $X$ and fix $U$ such that there is a 
surjective morphism $J_U:=\Jac(X_U) \twoheadrightarrow A$ defined over $\Q$. Let $\xi_U$ be the 
normalized Hodge class in $X_U$ and define 
\[
\pi_A:=\dirlim_U\Hom_{\xi_U}^0(X_U,A),
\]
where $\Hom_{\xi_U}^0(X_U,A)$ denotes morphisms of $\Hom(X_U,A)\otimes_\Z\Q$ defined by using 
the Hodge class $\xi_U$ as a base point. Since, by the universal property of Jacobians, every 
morphism $X_U\rightarrow A$ factors through $J_U$, we also have 
\[
\pi_A:=\dirlim_U\Hom_{\xi_U}^0(J_U,A),
\]
where $\Hom_{\xi_U}^0(J_U,A):=\Hom(J_U,A)\otimes_\Z\Q$. For any $\varphi\in \pi_A$ and any 
point $P\in X_U(H)$, where $H/\Q$ is a field extension, we then see that $P(\varphi):=\varphi(P)\in A(H)$.  

Let $K/\Q$ be an imaginary quadratic field and assume there exists an embedding $\psi:K\hookrightarrow B$; 
this is equivalent to say that all primes dividing $\Delta$ are inert or ramified in $K$. Define $X^{K^\times}$ 
to be the subscheme of $X$, defined over $\Q$, consisting of fixed points under the canonical action by 
left translation of $\hat\psi:\hat K^\times\hookrightarrow \hat B^\times$. The subscheme $X^{K^\times}$ 
is independent up to translation of the choice of $\psi$. We will often omit the reference to $\psi$, 
viewing $K$ simply as a subfield of $B$. Recall that the theory of complex multiplication shows that 
every point in $X^{K^\times}(\bar\Q)$ is defined over $K^{\mathrm{ab}}$, the maximal abelian extension 
of $K$, and that the Galois action is given by the left translation under the reciprocity map. 
Fix a point $P\in X^{K^\times}(K^\mathrm{ab})$. This amounts to choose a point $P_U$ for all open 
compact subgroups $U$, satisfying the condition that $\pi_{U_1,U_2}(P_{U_1})=P_{U_2}$. 

Let $d\tau$ denote the Haar measure of $\Gal(K^\mathrm{ab}/K)$ of total mass equal to $1$ 
and fix a finite order character $\chi:\Gal(K^\mathrm{ab}/K)\rightarrow F_\chi^\times$, where 
$F_\chi=\Q(\chi)$ is the finite field extension of $\Q$ generated by the values of $\chi$. Define 
\[
P_\chi(\varphi):=\int_{\Gal(K^\mathrm{ab}/K)}\varphi(P^\tau)\otimes\chi(\tau)d\tau.
\]
This is an element in $A(K^\mathrm{ab})\otimes_MF_\chi$, where $M=\End_\Q^0(A):=\End_\Q(A)\otimes_{\Z}\Q$. 
This element can be essentially written as a finite sum: suppose that $P=(P_U)_U$, and each $P_U$ is defined 
over the abelian extension $H_U$ of $K$. Suppose that $\chi$ factors through $\Gal(H_U/K)$ for some $U$. 
Then the $F_\chi$-subspace of $A(H_U)\otimes F_\chi$ spanned by $P_\chi(\varphi)$ and 
\[
\sum_{\sigma\in \Gal(H_U/K)}\varphi(P)^\sigma\otimes\chi(\sigma)
\]
are the same. We also note that $P_\chi(\varphi)$ belongs to $(A(H_U)\otimes_\Z\C)^\chi$. 

Let $\eta_K$ be the quadratic character of the extension $K/\Q$. Suppose that $\chi$ satisfies 
the self-duality condition $\omega_A \cdot \chi_{|\A_\Q^\times} = 1$, where $(\cdot)_{|\A_\Q^\times}$ 
means restriction of the character $(\cdot)$ to the idele group $\A_\Q^\times$ and $\omega_A$ is 
the central character of the automorphic representation $\pi_A$. We assume for simplicity that 
$\omega_A$ is trivial, and therefore $\chi_{|\A_\Q^\times}=1$. For any place $v$ of $\Q$, let 
$\varepsilon(1/2,\pi_{A,v},\chi_v)\in\{\pm 1\}$ be the sign of the functional equation with 
respect to its center of symmetry $s = 1/2$ of the local representation $\pi_{A,v}\otimes\chi_v$. 
Define the set
\[
\Sigma(A,\chi)=\left\{v \text{ place of $\Q$}: \varepsilon(1/2,\pi_{A,v},\chi_v)\neq \eta_{K,v}(-1)\right\}.
\]

\begin{prop}
 The real place $\infty$ belongs to the set $\Sigma(A,\chi)$, and every finite prime 
 $p\in \Sigma(A,\chi)$ divides the conductor of $A$.
\end{prop}
\begin{proof}
 According to \cite[Section 1]{CornutVatsal}, the real place $\infty$ belongs to the set 
 $\Sigma(A,\chi)$ if $\chi_{\infty} = 1$ and $\pi_{A,\infty}$ is the holomorphic discrete 
 series (of weight at least $2$). The first condition is true by our assumptions, while 
 the second one holds because $\pi_A$ is the automorphic representation attached to an 
 abelian variety. 
 On the other hand, also from {\em loc. cit.} we know that if $p$ is a finite prime in 
 the set $\Sigma(A,\chi)$, then $\pi_{A,p}$ is either special or supercuspidal, and 
 therefore $p$ must divide the conductor of $A$. 
\end{proof}

\begin{remark}
 If $p$ is a finite prime belonging to $\Sigma(A,\chi)$, one also knows that 
 $K_p := K \otimes_{\Q}\Q_p$ is a field. In particular, if $B$ is an indefinite quaternion 
 algebra whose ramification set is supported in $\Sigma(A,\chi)$, then $K$ splits $B$ 
 (i.e. $K$ embeds as a maximal subfield of $B$).
\end{remark}

Let $\varepsilon(1/2,\pi_A,\chi)$ be the sign of the functional equation with 
respect to its center of symmetry $s=1/2$ of the global representation $\pi_{A}\otimes\chi$. Then 
\[
\varepsilon(1/2,\pi_A,\chi)=(-1)^{|\Sigma(A,\chi)|}.
\]

Recall our assumption that the central character $\omega_A$ of $\pi_A$ is trivial, and let 
now $\chi$ be a character of $\Gal(K^\mathrm{ab}/K)$. Suppose that $\chi$ factors through 
$\Gal(H_c/K)$ where $H_c$ is the ring class field of conductor $c$; if there is no $c'\mid c$ 
such that $\chi$ factors through $\Gal(H_{c'}/K)$, we say that \emph{$\chi$ has conductor $c$}; 
if $\chi$ factors through $\Gal(H_c/K)$, then the conductor of $\chi$ divides $c$. 

Suppose we have a character $\chi$ of conductor dividing the positive integer $c$ and a Heegner 
point $P_c$ of conductor $c$ in $X(H_c)$. For any $\varphi\in \pi_A$ define 
\[
P_{c,\varphi}^\chi:=\sum_{\sigma\in \Gal(H_c/K)} \varphi(P_c^\sigma)\otimes\chi(\sigma).
\]
If $|\Sigma(A,\chi)|$ is odd, by \cite[Theorem 1.3.1]{YZZ} one can choose $\varphi$ such that  
\begin{equation}\label{YZZ}
P_{c,\varphi}^\chi \neq 0 \text{ in } (A(H_c)\otimes_\Z\C)^\chi\Longleftrightarrow L'(\pi_A,\chi,1/2)\neq 0.
\end{equation}
From now on, we fix such a $\varphi$ and write simply $P_{c}^\chi$ for $P_{c,\varphi}^\chi$. 

\subsection{Euler systems and BSD conjecture} 

Before discussing our applications to the BSD conjecture, we recall the following result, which in 
this general form is due to Nekov\'a\v r \cite{Nek}. 

\begin{thm}[Nekov\'a\v r]\label{nekovar}
Suppose that $A/\Q$ is a modular abelian variety of dimension $d$. Fix an imaginary quadratic field 
$K$ and an anticyclotomic character $\chi$ factoring through $H_c$ for some integer $c\geq 1$ such 
that the cardinality of $\Sigma(A,\chi)$ is odd. Let $B$ be the indefinite quaternion algebra of discriminant 
equal to the product of finite primes in $\Sigma(A,\chi)$. Assume that $A$ does not acquire CM over 
any imaginary quadratic field contained in $H_c$, and that there exists 
\begin{enumerate} 
\item an order $\cR$ of $B$ with an uniformization $J_\cR = \Jac(X_{\cR}) \twoheadrightarrow A$ defined over $\Q$, and
\item a Heegner point $P_c$ in $X_{\cR}(H_c)$. 
\end{enumerate} 
Then the following implication holds: 
\[
L'(\pi_A,\chi,1/2)\neq 0 \Longrightarrow \dim_\C\left(A(H_c)\otimes_\Z\C)^\chi\right) = d.
\] 
\end{thm}

We first observe that if the ramification set of the quaternion algebra $B$ coincides with 
$\Sigma(A,\chi) - \{\infty\}$, then there always exists a uniformization $J_U = \Jac(X_U) \twoheadrightarrow A$ 
for some open compact subgroup $U$ of $\hat B^\times$; so in (1) we are asking that this $U$ is 
associated with an order.

\begin{thm}\label{Main_Thm1} 
Fix the following objects: 
\begin{enumerate} 
\item a modular abelian variety $A/\Q$ of dimension $d$ and conductor $N^d$,
\item an imaginary quadratic field $K$ and 
\item an anticyclotomic character $\chi$ factoring through the ring class field $H_c$ of $K$ of 
conductor $c\geq 1$ such that the cardinality of $\Sigma(A,\chi)$ is odd.
\end{enumerate} 
Let $B$ denote the indefinite quaternion algebra of discriminant $\Delta$ equal to the product of 
all the finite primes in $\Sigma(A,\chi)$. Then there exists an order $\cR$ of type 
$T = (\Ne;\Nc;\{(L_p,\nu_p)\}_{p\mid\Delta})$ in $B$ and a Heegner point 
$P_{c'}\in X_\cR(H_{c'})$ with $c\mid c'$ such that: 
\begin{enumerate}
\item $A$ is uniformized by $X_{\cR}$, hence there is a surjective morphism $J_\cR \twoheadrightarrow A$ defined over $\Q$;  
\item $N$ divides the level $\Ne\cdot\Nc^2\cdot\prod_{p\mid\Delta}p^{\nu_p}$ of $\cR$; 
\item $c$ divides $c'$. 
 \end{enumerate}
 \end{thm}

\begin{proof} 
The problem is local, being equivalent to the existence of optimal local embeddings for all primes $\ell$. 
Fix the order $\cR_\mathrm{min}$ of type $T_\mathrm{min}=(\Ne;\Nc;\{(L_p,\nu_p')\}_{p\mid \Delta})$ and level 
$N_{\mathrm{min}}=\Ne\Nc^2N_{\Delta}$ as in the proof of Proposition \ref{propJL}, choosing the integers $\Ne$ 
and $\Nc$ such that $\Nc$ is divisible only by primes $p$ which are inert in $K$ and do not divide $c$. 

For primes $p \mid \Ne\Nc$ which are split in $K$, one knows that the set of local optimal embeddings of 
the required form is non-empty (cf. Lemma \ref{LemmaC0}).

Fix until the end of the proof a prime $p\mid N_{\mathrm{min}}$ which is inert or ramified in $K$. Let $m$ be 
the $p$-adic valuation of $c$ and set $n:=\val_p(N_{\mathrm{min}})$. If $p$ divides $\Nc$, then we 
can apply Lemma \ref{LemmaC2} and show that the maximal order $R_{c}\otimes_\Z\Z_p$ embeds optimally into  $\cR_\mathrm{min}\otimes_\Z\Z_p$. So suppose from now on that $p$ does not divide $\Nc$. 

Suppose first that $p\not\in\Sigma(A,\chi)$. If $m\geq n/2$ (unramified case) or $m\geq (n-1)/2$ (ramified 
case) then Lemma \ref{lemmaC1} shows that the set of local optimal embeddings of $R_c\otimes_\Z\Z_p$ into  $\cR_\mathrm{min}\otimes_\Z\Z_p$ is non-empty. If these conditions do not hold, replacing $m$ by $m'$ such 
that $m'\geq n/2$ (unramified case) or $m'\geq (n-1)/2$ (ramified case) then the local order $R_{c'}\otimes_\Z\Z_p$ 
with $c' = c\cdot p^{m'-m}$ embeds optimally into $\cR_{\mathrm{min}}\otimes_\Z\Z_p$. 

Suppose now that $p\in\Sigma(A,\chi)$. Take any pair $(m',n')$ satisfying the following condition: 
\begin{itemize}
\item If $n$ is odd, then $n'=2m'+1$;
\item if $n$ is even, then $n'=2m'$ if $p$ is inert in $K$ whereas $n'=2(m'+1)$ if $p$ ramifies in $K$.
\end{itemize} 
Choose also the pair $(m',n')$ so that $m'\geq m$ and $n'\geq n$. Comparing with the results recalled in 
\S \ref{appA}, we see that the set of optimal embeddings of the local quadratic order of conductor $p^{m'}$ 
into the local quaternion order $R_{n'}(L_p) \subseteq \mathcal R_{\mathrm{min}}\otimes_{\Z}\Z_p$ of type 
$(L_p,n')$ is non-zero.
\end{proof}

\begin{corol} 
Let $A/\Q$, $K$ and $\chi$ be as in the previous theorem. If $A$ does not acquire CM over 
any imaginary quadratic field contained in $H_c$ 
 and $L'(\pi_A,\chi,1/2)\neq 0$, then 
$\dim_\C\left(A(H_c)\otimes_\Z\C)^\chi\right) = d$. 
\end{corol} 
\begin{proof}
Let $\cR$ and $c'$ be as in the statement of Theorem \ref{Main_Thm1} and apply Theorem \ref{nekovar}, viewing $\chi$ as a character 
of $\Gal(H_{c'}/K)$ via the canonical projection $\Gal(H_{c'}/K) \to \Gal(H_c/K)$. 
\end{proof}

\subsection{Proof of Theorem A} 

Theorems \ref{nekovar} and \ref{Main_Thm1}, although giving an Euler System which is sufficient for the proof of 
our main result in the introduction, are not completely satisfying in the sense that they are not effective in the 
computation of the order $\mathcal R$. Suppose we are in the situation of the theorem, so that we are given a modular 
abelian variety $A/\Q$, an imaginary quadratic field $K$ and an anticyclotomic character $\chi$ of conductor $c$ such 
that $\Sigma(A,\chi)$ has odd cardinality. Then we would like to have Heegner points in $\cR_\mathrm{min}$ for a 
choice of minimal parametrization $J_{\cR_\mathrm{min}}\twoheadrightarrow A$ described in Proposition \ref{propJL}, 
or at least of level $\cR$ with $\cR\subseteq\cR_\mathrm{min}$. We begin by discussing a couple of examples.

\begin{example} \label{ex6.5} 
Let $A=E$ be an elliptic curve of conductor $N=p^2q$, with $q$ and $p$ odd distinct primes both inert in $K$. 
Assume that the automorphic representation $\pi_E$ attached to $E$ is supercuspidal at $p$. Let $B$ be the 
quaternion algebra of discriminant $pq$, $\mathcal R =\mathcal R_{\mathrm{min}}$ be the Hijikata--Pizer--Schemanske 
order $\cR=\cR_\mathrm{min}$ of level $N=p^2q$ (and type $(1;1;\{(L_p,2),(L_q,1)\})$, for any choice of quadratic 
ramified extension $L_p/\Q_p$), and let $X_{\cR}$ be its associated Shimura curve. Note that $K$ embeds into 
$B$ because both $p$ and $q$ are inert in $K$. 

Consider first the case of the trivial character $\mathbf{1}$. Then 
$\varepsilon_p(E/K,\mathbf{1})=\mathbf{1}_p(-1)=+1$ (\cite[(5.5.1)]{Deligne2}), 
and therefore $p\not\in\Sigma(E,\mathbf{1})$. So, if the sign of the functional equation of $E/K$ 
is $-1$ then $\Sigma(E,\mathbf{1})=\{\infty\}$ (because any prime in $\Sigma(E,\mathbf{1})$ must 
divide $N$). Note also that, even if $f$ admits a JL lift to $X_{\cR}$, there are no Heegner 
points of conductor $1$ in this curve (cf. \S \ref{appA}, case (1c)). 

For a non-trivial character $\chi$ of conductor $p$, again by case (1c) in \S \ref{appA}, we see that 
in this case Heegner points of conductor $c=p$ do exist in $X_{\cR}$. Therefore, if $\Sigma(E,\chi)$ is odd, 
then the set of Heegner points of conductor $c=p$ in $X_{\cR}$ is non-empty. 

Finally, consider the case of non-trivial conductor $p^m$ with $m\geq 2$. By \cite[p. 1299]{Tunnell} 
we know that $\varepsilon_p(E/K,\chi) = +1$, and therefore, as in the case of the trivial character, if the sign 
of the functional equation of $E/K$ is $-1$ then $\Sigma(E,\mathbf{1}) = \{\infty\}$. Also note that, again as 
in the case of the trivial character, there are no Heegner points of conductor $p^m$ in $X_{\cR}$. 
\end{example}

\begin{example} \label{ex6.6} 
As in the above example, let $A=E$ be an elliptic curve of conductor $N=p^2q$, with $q$ and $p$ 
odd distinct primes and suppose that $q$ is inert and $p$ is ramified in $K$. Identify the Weil--Deligne 
group $W_{\Q_p}$ of $\Q_p$ with $\Q_p^\times$ via the reciprocity map $r_{\Q_p}$, normalized in 
such a way that $r_{\Q_p}(a)$ acts on $\bar\F_p$ by the character $x\mapsto x^{|a|}$, where 
$|\cdot |=|\cdot|_p$ is the $p$-adic absolute value satisfying $|p|=p^{-1}$. Assume that the 
automorphic representation $\pi_E$ attached to $E$ is supercuspidal at $p$, and write it as
$\pi_{E,p}=\mathrm{Ind}_{W_F}^{W_{\Q_p}}(\psi)$ where $F/\Q_p$ is a quadratic extension with 
associated character $\eta$ and $\psi:W_F^\mathrm{ab}\rightarrow\C^\times$ is a quasi-character 
not factoring through the norm map; then we have $\eta\psi=|\cdot|^{-1}$ as quasi-characters of 
$\Q_p^\times$. The above conditions force $\psi$ to have conductor equal to $1$, $p\equiv 3\mod 4$ 
and $\psi_{|\Z_p^\times}=\eta$ (\cite[Cor. 3.1]{Pacetti}). Consider the quaternion algebra $B$ of 
discriminant $pq$, the Hijikata--Pizer--Schemanske order $\cR$ of level $p^2q$ as in the previous example 
and its associated Shimura curve $X_\cR$. Again $\pi_E$ admits a Jacquet--Langlands lift to $X_\cR$. Take 
a character $\chi$ of conductor $p^m$, $m\geq 1$. Then $\varepsilon_p(E/K,\chi)=1$ if $m\geq 4$ by 
\cite[p. 1299]{Tunnell}. 
If $p\equiv 3\mod 4$, then $\eta_{K,p}(-1)=-1$ and therefore $p\in\Sigma(E,\chi)$. Assuming that 
$\Sigma(E,\chi)$ is odd, we see that $q\in\Sigma(E,\chi)$ too. However, there are no Heegner points 
of conductor $p^m$ with $m\geq 1$ in $X_\cR$.
\end{example}
%
%

The above examples motivate our discussion below, leading to the proof of (a slightly refined version of) 
Theorem A in the Introduction. 

Fix for the rest of the article an elliptic curve $E/\Q$ of conductor $N$, an imaginary quadratic field $K$ 
of discriminant $-D$ and a ring class character $\chi$ of conductor $c$ of $K$. Let $\Delta$ be the product 
of the finite primes in $\Sigma(A,\chi)$, which is assumed to have odd cardinality, and let $B$ be the 
quaternion algebra of discriminant $\Delta$. Fix also $\mathcal R:=\cR_\mathrm{min}$ to be the minimal order 
of type $T_{\mathrm{min}}=(\Ne;\Nc;\{(L_p,\nu_p)\}_{p\mid \Delta})$ as in Proposition \ref{propJL}, on which 
the Jacquet--Langlands lift to $B$ of the newform $f \in S_2(\Gamma_0(N))$ associated with $E$ is realized, and 
let $N_{\mathcal R} = \Ne\Nc^2N_{\Delta}$ be its level. We can further assume the (coprime) integers $\Ne$ and 
$\Nc$ satisfy that, for every prime $p$ dividing $\Ne\Nc$,
\[
 p \mid \Nc \text{ if and only if } p \text{ is inert in } K \text{, } \val_p(N) \text{ is even and } p \nmid c.
\]

From now on, we shall make the following assumption on $N$. Observe that under the hypothesis that neither 
$2^3$ nor $3^4$ divide $N$, as in the statement of Theorem A, the assumption below is obviously satisfied.

\begin{assumption}\label{assumption2N}
Let $f\in S_2(\Gamma_0(N))$ be the newform
attached to the elliptic curve $E$ by modularity. 
 With the above notations, the following holds:
 \begin{enumerate}
  \item If $2^3 \mid \Ne$, then either $2$ splits in $K$ or $\val_2(\Ne)$ is odd and 2 is inert in $K$.
  \item  If $\Delta$ is even and $2^3\mid N$, then $2$ is inert in $K$, and if in addition $\pi_E$ is
supercuspidal at 2 then $\val_2(N)$ is odd.  
    \item The $2$-component $\pi_{f,2}$ of the automorphic representation attached to $f$ 
    has minimal Artin conductor among its twists by quasi-characters of $\Q_2^\times$. 
  \item If $\val_3(\Ne) = 4$ and $3$ is inert in $K$, then $\val_3(c)\neq 1$.
 \end{enumerate}
\end{assumption}

\begin{remark}\label{rem-ass}
If $f\in S_2(\Gamma_0(N))$ and $\val_2(N)$ is even and greater or equal than $4$, then $f$ is a twist of a modular 
form of lower level by \cite[Thm. 3.9]{HPS2}. Conditions (1) and (2) in the above assumption rule out these cases; 
in fact, one expects to treat these cases via different methods. If $f=g\otimes\chi$, then one expects to construct 
points on the modular abelian varieties attached to $g$, and then, using twisting techniques, to construct points on 
the elliptic curve. It seems possible that condition (3) can be treated by similar considerations. 
\end{remark}

Write also $R_c$ for the order of conductor $c$ in $K$ of conductor $c$ as usual. Our goal now is to 
investigate under which conditions $R_c$ embeds optimally into $\cR_{\mathrm{min}}$. And in those cases 
where this does not happen, we must find a suitable suborder of $\cR_\mathrm{min}$ such that $R_c$ does 
optimally embed into it. 

The problem is clearly local, and it suffices to study it at those primes dividing $N$. So fix from now 
on a prime $p\mid N$, and set the following notations. We write $m := \val_p(c)$ for the $p$-adic valuation 
of $c$, and $n := \val_p(N_{\mathcal R})$ for that of $N_{\mathcal R}$. By the discussion prior to 
Proposition \ref{propJL}, observe that $n$ coincides with $\val_p(N)$ if $p$ is odd (cf. also Remark 
\ref{remarkJL}). And if $p=2$, under Assumption \ref{assumption2N} the only cases with $n\neq \val_2(N)$, and 
for which this unequality is relevant in our discussion, are those in Lemma \ref{lemma:St-inert} below. In 
view of this, except for Lemma \ref{lemma:St-inert} we may use that $n=\val_p(N)$ without further explicit 
mention. Then we denote by $\mathcal E_p(m,n)$ the set of (local) optimal embeddings of $R_c \otimes_{\Z} \Z_p$ 
into $\cR \otimes_{\Z} \Z_p$. Recall that the conditions that characterize the non-emptiness 
of $\mathcal E_p(m,n)$, in each of the possible cases, have been collected in Section \ref{opt-emb-sec}. 
If $p$ is not split in $K$, then we write $\chi_p$ for the component of $\chi$ at the unique prime of $K$ 
above $p$. In that case, notice that $m = c(\chi_p)$, the (exponent of the) conductor of $\chi_p$.

If the prime $p$ divides $\Ne\Delta$ and $n'\geq n$ is an integer, then we define a suborder $\cR'$ of $\cR$ 
in the following way. If $p\mid \Ne$, then we define $\cR'$ to be locally equal to $\cR$ at every prime 
distinct from $p$, and such that $\cR' \otimes_{\Z}\Z_p$ is the Eichler suborder of level $p^{n'}$ (hence 
index $p^{n'-n}$) in $\mathcal R \otimes_{\Z}\Z_p$. In particular, $\val_p(N_{\mathcal R'}) = n'$. And if 
$p\mid \Delta$, then $\mathcal R'$ is obtained from $\mathcal R$ by replacing the local data $(L_p,n)$ at 
$p$ in the type of $\mathcal R$ by the data $(L_p,n')$. Besides, given an integer $m'\geq m$, we denote by 
$R_{c'}$ the suborder of $R_c$ of conductor $c'=cp^{m'-m}$. Finally, given two integers $m'\geq m$ and 
$n'\geq n$, we will simply denote by $\mathcal E_p(m',n')$ the set of (local) optimal embeddings of 
$R_{c'} \otimes_{\Z} \Z_p$ into $\cR' \otimes_{\Z} \Z_p$.

%

First we consider the case where $p$ does not belong to the set $\Sigma(E,\chi)$, so that $B$ is split at $p$.
 
\begin{lemma} 
If $p\not\in\Sigma(E,\chi)$, then $\mathcal E_p(m,n)\neq\emptyset$.
\end{lemma}
\begin{proof}
First observe that if $p\not\in\Sigma(E,\chi)$ then $p$ divides $\Ne\Nc$. Having said this, notice that if 
$p\mid \Nc$ then $\mathcal E_p(m,n)$ is non-empty by Lemma \ref{LemmaC2} (because if $p \mid \Nc$ then $m=0$). 
So we assume for the rest of the proof that $p$ divides $\Ne$. By our choice of $\Ne$ and $\Nc^2$, 
we shall distinguish three cases:
\begin{enumerate}
\item $p$ is split in $K$;
\item $p$ is inert or ramified in $K$ and $n$ is odd; 
\item $p$ is inert or ramified in $K$, $n$ is even and $p\mid c$.
\end{enumerate} 

If $p$ is in case (1), then Lemma \ref{LemmaC0} implies that $\mathcal E_p(m,n)$ is non-empty. Suppose 
that $p$ is in case (2), and assume first that $n=1$, which is the only possible value if $p\geq 5$. If 
$p$ ramifies in $K$, then $\mathcal E_p(m,n) \neq \emptyset$ by part (ii) in Lemma \ref{lemmaC1}. If $p$ 
is inert in $K$, then we split the discussion according to whether $p\nmid c$ or $p\mid c$. In the former 
case, $\varepsilon_p(E/K,\chi) = \varepsilon_p(E/K,\mathrm{1}) = -1$ but $\eta_{K,p}(-1)=1$, thus $p$ should 
be in $\Sigma(E,\chi)$, contradicting our hypotheses. And in the latter case, we have $m\geq 1$ and 
therefore $2m \geq n = 1$, hence Lemma \ref{lemmaC1} shows that $\mathcal E_p(m,n)\neq\emptyset$. Thus 
we are left with the cases where $p=2$ or $3$ and $n=\val_p(N)>1$ is odd.

\begin{itemize}
 \item If $p=3$, then $n$ can be either $3$ or $5$. Then $\pi_{E,3}$ is supercuspidal induced from a 
 quasicharacter $\psi$ of conductor $n-1$ of a ramified quadratic extension $F_3$ of $\Q_3$. If $3$ is inert 
 in $K$, we know on the one hand by Lemma \ref{lemmaC1} that $\mathcal E_3(m,n)\neq \emptyset$ if and only if 
 $m \geq n/2$, hence if and only if $m > (n-1)/2$. On the other hand, being $3$ inert in $K$ the assumption 
 that $3\not\in \Sigma(E,\chi)$ tells us that $\varepsilon_3(E/K,\chi) = 1$, and by \cite[Prop. 2.8]{Tunnell} 
 this holds if and only if $m > (n-1)/2 = 1$. Thus it follows that $\mathcal E_3(m,n)\neq \emptyset$. Now 
 suppose that $3$ ramifies in $K$. Then $\eta_{K,3}(-1)=-1$, hence $\varepsilon_3(E/K,\chi)=-1$ because 
 $3\not\in \Sigma(E,\chi)$. By Lemma \ref{lemmaC1} we have that $\mathcal E_3(m,n)\neq \emptyset$ if and 
 only if $m\geq (n-1)/2$. We next show that $\varepsilon_3(E/K,\chi)=-1$ implies that $m\geq (n-1)/2$. If 
 $F_3\neq K_3$, \cite[Cor. 1.9.1]{Tunnell} implies that $\varepsilon_3(E/K,\chi) = (-1)^{c(\psi\chi_3)}$. If 
 it were $m=c(\chi_3) < c(\psi) = n-1$, we would have $c(\psi\chi_3)=c(\psi) = n-1$. Since $n$ is $3$ or $5$, 
 in both cases we would have $\varepsilon_3(E/K,\chi)=1$, thus it must be $m > n-1 > (n-1)/2$, and therefore 
 $\mathcal E_3(m,n)\neq \emptyset$. If in contrast $F_3 = K_3$, we can apply \cite[Prop. 1.10]{Tunnell} 
 instead, which tells us that $\varepsilon_3(E/K,\chi) = (-1)^{c(\psi^{\tau}\chi_3)+c(\psi\chi_3)}$, where 
 $\tau$ is the non-trivial automorphism in $\Gal(K_3/\Q_3)$. Similarly as before, observe now that if 
 $c(\psi)\neq c(\chi_3)$, then $c(\psi^{\tau}\chi_3)=c(\psi\chi_3)=\max(c(\psi),c(\chi_3))$, and as a 
 consequence $\varepsilon_3(E/K,\chi) = 1$. Thus it must be $m=c(\chi_3)=c(\psi)=n-1 >(n-1)/2$, and we 
 conclude that $\mathcal E_3(m,n)\neq \emptyset$. 
 
 \item If $p=2$, then $n$ can be either $3$, $5$ or $7$. And by Assumption \ref{assumption2N} i), we may suppose 
 that $2$ is inert in $K$, so that $2\not\in \Sigma(E,\chi)$ implies that $\varepsilon_2(E/K,\chi) = 1$.
 
 Suppose first that $n=3$. Then \cite[Prop. 3.7]{Tunnell} implies that $m\geq 2$, and then by part (i) 
 of Lemma \ref{lemmaC1} we deduce that $\mathcal E_2(m,n)\neq \emptyset$. 
  
 Suppose now that $n=5$. Now $\pi_{E,2}$ is supercuspidal induced from a quasicharacter of conductor $3$ on a 
 ramified extension of $\Q_2$ with discriminant valuation $2$. If the conductor of $\chi$ were $m < 3$, then 
 \cite[Cor. 1.9.1]{Tunnell} would imply that $\varepsilon_2(E/K,\chi) = -1$, thus we deduce that $m\geq 3$. And 
 then by part (i) of Lemma \ref{lemmaC1} we conclude that $\mathcal E_2(m,n)\neq \emptyset$. 
 
 If $n=7$, then $\pi_{E,2}$ is supercuspidal of exceptional type, and its conductor is minimal with respect to 
 twist. Then $\varepsilon_2(E/K,\chi) = 1$ implies, by \cite[Lemma 3.2]{Tunnell}, that $m\geq 4$. But then we deduce 
 that $\mathcal E_2(m,n)\neq \emptyset$ thanks to Lemma \ref{lemmaC1}, part (i)  
 (here, and here only, we use condition iii) in Assumption \ref{assumption2N}).\end{itemize}

Finally, suppose that $p$ is in case (3). Again let us start with the case $n = 2$, which is the only possible 
case if $p\geq 5$. Since $m = \val_p(c) \geq 1$, we see that $2m \geq n$, hence Lemma \ref{lemmaC1} implies that 
$\mathcal E_p(m,n) \neq \emptyset$, regardless $p$ is inert or ramified in $K$. Thus we are again left with the 
cases where $p=2$ or $3$ and $n = \val_p(N) > 2$ is even. 

By Assumption \ref{assumption2N} i), the case $p=2$ does not arise, so we assume that $p=3$. Then the only possible 
value for $n$ is $4$. In this case $\pi_{E,3}$ is supercuspidal induced from a quasicharacter of conductor $2$ of 
the unramified quadratic extension of $\Q_3$. If $3$ is ramified in $K$, then $\eta_{K,3}(-1)=-1$, and therefore 
$\varepsilon_3(E/K,\chi)=-1$. By \cite[Cor. 1.9.1]{Tunnell}, if $m=1$ then we would have $\varepsilon_3(E/K,\chi)=1$, 
hence we see that $m\geq 2$. On the other hand, part (ii) in Lemma \ref{lemmaC1} tells us that $\mathcal E_3(m,n)\neq \emptyset$ 
if and only if $m\geq 2$. Thus we deduce that indeed $\mathcal E_3(m,n)\neq \emptyset$.
 
If in contrast $3$ is inert, then Assumption \ref{assumption2N} iii) implies that $m\geq 2$, and by part (i) in Lemma 
\ref{lemmaC1} we conclude that $\mathcal E_3(m,n)\neq \emptyset$.
%
\end{proof}

Next we will deal with the case that $p \in \Sigma(E,\chi)$, or equivalently $p \mid \Delta$. This means that  
$\varepsilon_p(E/K,\chi) = -\eta_{K,p}(-1)$. So if $p$ is odd, then  
\[
\varepsilon_p(E/K,\chi)=\begin{cases}
-1 & \text{ if $p$ is inert in $K$},\\
-1 & \text{ if $p$ is ramified in $K$ and $p\equiv 1\mod 4$},\\
1 & \text{ if $p$ is ramified in $K$ and $p\equiv 3 \mod 4$}.
\end{cases} 
\]

Let $\pi_E$ be the automorphic representation attached to $E$, and $\pi_{E,p}$ be its $p$-th component. The 
(exponent of the) conductor of $\pi_{E,p}$ is $\val_p(N)$. 

We will split our discussion into distinct lemmas, to distinguish between the cases where $\pi_{E,p}$ is 
supercuspidal or Steinberg. If $\pi_{E,p}$ is supercuspidal, then it is well-known that $\val_p(N)\geq 2$. 
For $p\geq 5$ this means that $\val_p(N)=2$, whereas for $p=3$ (resp. $p=2$) we have $2\leq \val_3(N)\leq 5$ 
(resp.  $2\leq \val_2(N)\leq 8$). Besides, if $\pi_{E,p}$ is Steinberg, then $\val_p(N)$ can only be $1$ 
or $2$ if $p$ is odd, whereas if $p=2$ then $\val_2(N) \in \{1,4,6\}$. However, the reader should keep in mind 
that under Assumption \ref{assumption2N}, some of the previous cases with $p=2$ do not appear in our discussion.


%
%
%
%

\begin{lemma}
If $p\in\Sigma(E,\chi)$, $\pi_E$ is supercuspidal at $p$ and $p$ is inert in $K$ then there 
exists $n'\geq n$ such that $\mathcal E_p(m,n')\neq\emptyset$. 
\end{lemma}
\begin{proof} 
The assumptions $p\in\Sigma(E,\chi)$ and $p$ inert in $K$ imply that $\varepsilon_p(E/K,\chi) = -1$. Suppose first that 
$p$ is odd. We have the following cases:
\begin{enumerate}
 \item $n=2$. If $m=0$, then $\varepsilon_p(E/K,\chi)=1$ by \cite[(5.5.1)]{Deligne2}, so we may assume 
 that $m\geq 1$. But then defining $n':=2m \geq n$ we conclude by case (1c) in \S \ref{appA} that 
 $\mathcal E_p(m,n')\neq\emptyset$.
 
 \item $p=3$ and $n > 2$. In this case $n = \val_3(N)$ can be either $3$, $4$ or $5$. If $n\neq 4$, then 
 $\pi_{E,p}$ is induced from a quasicharacter $\psi$ of conductor $n-1$ of a ramified quadratic 
 extension of $\Q_3$. On the one hand, from \S \ref{appA} case (1a) we see that $\mathcal E_3(m,n)\neq\emptyset$ 
 if and only if $m\leq (n-1)/2$. And on the other hand, by \cite[Prop. 2.8]{Tunnell} one has that 
 $\varepsilon_3(E/K,\chi) = -1$ if and only if $m\leq (n-1)/2$. Thus we conclude that $\mathcal E_3(m,n)\neq\emptyset$.

 Suppose now that $n=4$. In this case, $\pi_{E,3}$ is induced from a quasicharacter $\psi$ of conductor $2$ of the 
 unramified quadratic extension of $\Q_3$. If $\chi_p$ is unramified, that is $m=0$, then by \cite[(5.5.1)]{Deligne2} 
 we would have $\varepsilon_p(E/K,\chi)=1$, hence it must be $m\geq 1$. However, if $m=1$ then one can use 
 \cite[Prop. 1.10]{Tunnell} (cf. also the proof of Proposition 3.5 in {\em loc. cit.}) to show that $\varepsilon_3(E/K,\chi) = 1$, 
 thus it follows that $m\geq 2$. But then for $n' := 2m \geq n$ we have that $\mathcal E_3(m,n')\neq \emptyset$ by case 
 (1c) in \S \ref{appA}.
\end{enumerate}

 Now we assume that $p=2$. Again we can split the discussion into cases.
\begin{itemize}
 \item[i)] First suppose $n = \val_p(N) = 2$. As above, if $m=0$ then $\varepsilon_p(E/K,\chi)=1$, hence it 
 must be $m\geq 1$. Letting $n':=2m \geq n$, case (2b) now ensures that $\mathcal E_p(m,n')\neq\emptyset$.
 
 \item[ii)] If $\val_p(N) > 2$, then Assumption \ref{assumption2N} ii) implies that $n$ is odd. If $n=3$, 
 on the one hand by \cite[Prop. 3.7]{Tunnell} we have that $\varepsilon_2(E/K,\chi) = -1$ if and only if 
 $m\leq 1$. And on the other hand, case (2f) in \S \ref{appA} tells us that $\mathcal E_2(m,n)\neq\emptyset$ 
 if and only if $m\leq 1$, thus we conclude that $\mathcal E_2(m,n)\neq\emptyset$. If $n$ is either $5$ or $7$, 
 again according to \S \ref{appA} case (2f) we see that if $m \leq (n-1)/2$ then $\mathcal E_2(m,n)\neq \emptyset$. 
 If not, defining $n':=2m+1$ we will have $\mathcal E_2(m,n')\neq \emptyset$.
\end{itemize}
This concludes the proof. 
\end{proof}


\begin{lemma}
If $p\in\Sigma(E,\chi)$, $\pi_E$ is supercuspidal at $p$ and $p$ is ramified in $K$ then 
there exists $n'\geq n$ such that $\mathcal E_p(m,n')\neq\emptyset$. 
\end{lemma}

\begin{proof}
As in the previous lemma, we assume first that $p$ is odd. We have the following cases:

\begin{enumerate}
 \item Suppose $n=\val_p(N)=2$. If $m=0$, we deduce from \S \ref{appA} (cases (1d) or (1e)) that $\mathcal E_p(m,n)\neq\emptyset$. 
 In contrast, if $m > 0$ the set $\mathcal E_p(m,n)$ is empty. But by virtue of \S \ref{appA}, case (1d) or (1e), 
 for $n':=2(m+1) > n$ we have $\mathcal E_p(m,n')\neq\emptyset$.
 
 \item Suppose that $p=3$ and $n=\val_3(N)\geq 3$. In this case, $3 \leq n \leq 5$. Assume first that $n=4$. In 
 this case, the quadratic extension $L_3/\Q_3$ is ramified. Up to replacing $L_3$ by the other quadratic ramified 
 extension, we might assume that $K_3\not\simeq L_3$. Then from case (1e) in \S \ref{appA} we see that 
 $\mathcal E_p(m,n)\neq\emptyset$ if and only if $m \leq 1$. If $m>1$, then we take $n':=2(m+1)$, and 
 again case (1e) in \S \ref{appA} tells us that $\mathcal E_p(m,n')\neq\emptyset$.
 
 Suppose now that $n$ is either $3$ or $5$. Then case (1b) in \S \ref{appA} shows that $\mathcal E_p(m,n)\neq\emptyset$ 
 if and only if $m = (n-1)/2$. Provided that $m\geq (n-1)/2$, defining $n' := 2m + 1 \geq n$ we obtain $\mathcal E_p(m,n')\neq\emptyset$ 
 as we want. Thus we need to prove that $m\geq (n-1)/2$. Since $3$ is ramified, we have $\eta_{K,3}(-1)=-1$, and therefore the 
 hypothesis that $3$ belongs to $\Sigma(E,\chi)$ implies that $\varepsilon_3(E/K,\chi) = 1$. Besides, we know that $\pi_{E,3}$ is 
 induced from a quasicharacter of conductor $n-1$ of a ramified quadratic extension of $\Q_3$. By \cite[Prop. 2.8]{Tunnell}, 
 $\varepsilon_3(E/K,\chi) = 1$ then implies that $m\geq (n-1)/2$ as we wanted.
\end{enumerate}

Now we deal with the case $p=2$. By Assumption \ref{assumption2N}, if it were $\val_2(N)>2$ then $2$ should be inert in $K$, thus 
we only need to consider the case $n=\val_2(N)=2$. By cases (2c), (2d) or (2e) in \S \ref{appA} we have that 
$\mathcal E_p(m,n)\neq\emptyset$ if and only if $m=0$. Notice first that $\pi_{E,p}$ is of minimal conductor among its twists, 
since supercuspidal representations have conductor $\geq 2$. Then, by virtue of \cite[Proposition 3.5]{Tunnell}, we see that 
$\varepsilon_p(E/K,\chi)$ is $+1$ (resp. $-1$) if and only if $m\geq 2$ (resp. $m < 2$). In view of this, if $\eta_{K,2}(-1) = 1$, 
the hypothesis that $2 \in \Sigma(E,\chi)$ implies that $m<2$. Since $m$ cannot be $1$, we deduce that $m=0$, and hence 
$\mathcal E_p(m,n)\neq\emptyset$ as desired. In contrast, assume that $\eta_{K,2}(-1) = -1$. Then it follows that $m\geq 2$. 
By defining $n':=2(m+1)\geq n$, we see by \S \ref{appA} (case (2c), (2d) or (2e)) that $\mathcal E_p(m,n')\neq\emptyset$.
\end{proof}

Next we consider the Steinberg case. Write $\pi_{E,p}=\mathrm{Sp}_2\otimes\psi$ where $\psi:W_{\Q_p}^\mathrm{ab}\rightarrow\C^\times$ 
is a quadratic character. By \cite[Prop. 1.7]{Tunnell}, we know that $\varepsilon(E,\chi)=-1$ if and only if 
$\chi_p^{-1} = \psi\circ\mathrm{Nr}$, where $x\mapsto \mathrm{Nr}(x)$ is the norm map from $K_p=K\otimes_\Q\Q_p$ to $\Q_p$.

\begin{lemma}\label{lemma:St-inert}
If $p\in\Sigma(E,\chi)$, $\pi_E$ is Steinberg at $p$ and $p$ is inert in $K$ then there 
exists $n'\geq n$ such that $\mathcal E_p(m,n')\neq\emptyset$. 
\end{lemma}
\begin{proof}
By the above discussion, $p\in\Sigma(E,\chi)$ if and only if $\chi_p^{-1} = \psi\circ\mathrm{Nr}$. We assume 
first that $p$ is odd, so that $n=\val_p(N)$ can be either $1$ or $2$. We split the discussion into subcases:

\begin{enumerate}
 \item $n=1$. Comparing with \S \ref{appA} (1a), we see that $\mathcal E_p(m,n)\neq\emptyset$ if and 
 only if $m=0$. On the other hand, $\psi$ is unramified and therefore if $p\in\Sigma(E,\chi)$ then $m=0$. 
 \item $n=2$. Looking now at \S \ref{appA} (1c), $\mathcal E_p(m,n)\neq\emptyset$ if and only if $m=1$. On 
 the other hand, $\psi$ is ramified with conductor equal to $1$, and therefore if $p\in\Sigma(E,\chi)$ then $m=1$.
\end{enumerate}

Assume now that $p=2$. Then $\val_2(N) \in \{1,4,6\}$.
 \begin{itemize}
 
 \item[i)] If $n=\val_2(N)=1$, the character $\psi$ is unramified, and then since $2 \in \Sigma(E,\chi)$ we 
 deduce that $m=0$. On the other hand, by case (2a) in \S \ref{appA} we also have 
$\mathcal E_p(m,n)\neq\emptyset$ if and only if $m=0$. Thus $\mathcal E_p(m,n)\neq\emptyset$ as we want. 
 
 \item[ii)] If $\val_2(N)=4$, then $\psi$ is ramified with conductor $2$. Since $\chi_p^{-1} = \psi \circ\mathrm{Nr}$, 
 it follows that also $m=2$. By the discussion before Proposition \ref{propJL}, $n=\val_2(N_{\mathcal R})$ 
 can be either $1$, $2$ or $3$. If $n$ is either $1$ or $3$, defining $n':=5 \geq n$ we see by case (2f) 
 in \S \ref{appA} that $\mathcal E_2(m,n')\neq \emptyset$. And if $n=2$, by case (2b) in \S \ref{appA} we 
 conclude that for $n':= 2m = 4 \geq n$ we have $\mathcal E_2(m,n')\neq \emptyset$.

 \item[iii)] If $\val_2(N)=6$, then $\psi$ is ramified with conductor $3$. Similarly as before, we deduce 
 that $\chi_p$ has also conductor $m=3$. From the discussion before Proposition \ref{propJL}, we know that 
 the possible values for $n=\val_2(N_{\mathcal R})$ are $1$, $2$, $3$, $4$ or $5$. When $n$ is either $1$, $3$ 
 or $5$, $L_2$ is the unramified quadratic extension of $\Q_2$, and defining $n':= 2m+1 = 7 \geq n$, case (2f) 
 in \S \ref{appA} implies that $\mathcal E_2(m,n')\neq \emptyset$. Finally, if $n=2$ or $n=4$, then we 
 may define $n':= 2m = 6 \geq n$ and now case (2b) shows that $\mathcal E_2(m,n')\neq \emptyset$. 
\end{itemize}
This concludes the proof.
\end{proof}


\begin{lemma}
If $p\in\Sigma(E,\chi)$, $\pi_E$ is Steinberg at $p$ and $p$ is ramified in $K$ then there exists $n'\geq n$ 
such that $\mathcal E_p(m,n')\neq\emptyset$. 
\end{lemma}
\begin{proof} 
Suppose first that $p$ is odd, so that $n=\val_p(N)$ is either $1$ or $2$. By the above discussion, 
$p\in\Sigma(E,\chi)$ if and only if either $\chi_p^{-1} = \psi\circ\mathrm{Nr}$ and $p\equiv 1\mod 4$, or 
$\chi_p^{-1} \neq \psi\circ\mathrm{Nr}$ and $p\equiv 3\mod 4$. We have the following cases:

\begin{enumerate}
\item $n = 1$. If $m=0$, by case (1b) in \S \ref{appA} we see that $\mathcal E_p(m,n)\neq\emptyset$. 
Otherwise, we can take $n':=2m+1$, and again case (1b) in \S \ref{appA} implies $\mathcal E_p(m,n')\neq\emptyset$. 

\item $n = 2$. Again, for $m=0$ we have $\mathcal E_p(m,n)\neq\emptyset$ by case (1d) or (1e) in \S \ref{appA}. If 
instead $m > 0$, then we consider $n':=2(m+1)$ and by applying \S \ref{appA}, case (1d) or (1e), we see that 
$\mathcal E_p(m,n')\neq\emptyset$.
\end{enumerate}

Now assume that $p=2$. Similarly as in the previous lemma, by Assumption \ref{assumption2N} we only need to deal with 
the case $n=\val_2(N)=1$. Then notice that $\psi$ is unramified. On the other hand, now the hypothesis that $2$ belongs 
to $\Sigma(E,\chi)$ implies that either $\eta_{K,2}(-1) = 1$ and $\chi_p^{-1} = \psi\circ\mathrm{Nr}$ or $\eta_{K,2}(-1) = -1$ 
and $\chi_p^{-1} \neq \psi\circ\mathrm{Nr}$. Having this into account, if $\eta_{K,2}(-1) = 1$ then the equality 
$\chi_p^{-1} = \psi\circ\mathrm{Nr}$ implies that $m=0$. By case (2a) in \S \ref{appA} it thus follows that 
$\mathcal E_p(m,n)\neq\emptyset$. And if $\eta_{K,2}(-1) = -1$, it could be the case that $m>0$. But in any case, defining 
$n':=2m+1 \geq n=1$ case (2g) in \S \ref{appA} implies that $\mathcal E_p(m,n')\neq\emptyset$.
\end{proof}

Combining the above lemmas, we obtain the following: 

\begin{thm}\label{main_thm}
Let $E/\Q$ be an elliptic curve of conductor $N$, $K$ be an imaginary quadratic field and $\chi$ be an 
anticyclotomic character of conductor $c$. Suppose that the set $\Sigma(E,\chi)$ has odd cardinality, so 
that $\varepsilon(E/K,\chi)=-1$ and hence $L(E/K,\chi,1)=0$. 
If Assumption \ref{assumption2N} holds, then the set of Heegner points in $E(H_c)$ is non-empty. And if 
further $E$ does not acquire CM over any imaginary quadratic field contained in $H_c$ 
and $L'(E/K,\chi,1)\neq 0$, then $\dim_{\C}\left(E(H_c)\otimes \C_{\Z} \right)^{\chi}=1$.
\end{thm} 
\begin{proof}
Let $B$ be the indefinite quaternion algebra ramified exactly at the finite primes in $\Sigma(E,\chi)$, 
and let $\cR_{\mathrm{min}}$ be the order in $B$ from Proposition \ref{propJL}. The above lemmas together 
imply that there is a suborder $\cR$ of $\cR_{\mathrm{min}}$ such that the set of Heegner points of 
conductor $c$ in $X_{\cR}(H_c)$ is non-empty. The Jacobian of $X_{\cR}$ uniformizes $E$ as well, hence 
the set of Heegner points of conductor $c$ in $E(H_c)$ is non-empty. By Theorem \ref{nekovar}, if $E$ 
does not acquire CM over any imaginary quadratic field contained in $H_c$, then 
$\dim_{\C}\left(E(H_c)\otimes \C_{\Z} \right)^{\chi}=1$.
\end{proof} 

We state now the above result in more restrictive, but maybe more attractive form, introducing 
a couple of definitions. 

\begin{defn}We say 
that a form $f\in S_2(\Gamma_0(N))$ is \emph{primitive} if $f\neq g\otimes\chi$ for any 
Dirichlet character $\chi$ and any $g\in S_2(\Gamma_0(M))$ with $M\mid N$ and $M\neq N$ 
(see \cite[Def. 8.6 ]{Pizer} for a similar terminology). 
\end{defn}
 
\begin{defn} Let $p$ be a prime. 
We say that $f\in S_2(\Gamma_0(N))$ \emph{has $p$-minimal Artin conductor} if the $p$-component $\pi_{f,p}$ 
of the automorphic representation $\pi_f$ attached to $f$ has minimal conductor among its twists by 
quasi-characters of $\Q_p^\times$; in other words, if we write $a(\pi_f)$ for the Artin conductor of the 
automorphic representation $\pi_f$, we require that $a(\pi_f)\leq a(\pi_f\otimes\chi)$ 
for all quasi-characters $\chi$ of $\Q_p^\times$.
\end{defn}

\begin{corol} \label{coro-final}
Let $E/\Q$ be an elliptic curve of conductor $N$, $K$ be an imaginary quadratic field and $\chi$ be an 
anticyclotomic character of conductor $c$. Suppose that the set $\Sigma(E,\chi)$ has odd cardinality. 
If the following conditions hold: 
\begin{enumerate}
\item $f$ is a primitive form; 
\item $f$ has $2$-minimal Artin conductor;
\item If $\val_3(\Ne) = 4$ and $3$ is inert in $K$, then $\val_3(c)\neq 1$;
\item If $\val_2(\Ne)\geq 3$, then $2$ is not ramified in $K$;
\end{enumerate}
then the set of Heegner points in $E(H_c)$ is non-empty. If further $E$ does not acquire CM over 
any imaginary quadratic field contained in $H_c$ and $L'(E/K,\chi,1)\neq 0$, then $\dim_{\C}\left(E(H_c)\otimes \C_{\Z} \right)^{\chi}=1$.
\end{corol}

\begin{proof}
We only need to remark that conditions i) and ii) in Assumption \ref{assumption2N} 
are satisfied if we ask that $f$ is primitive by \cite[Thm. 3.9]{HPS2}. 
\end{proof}

\subsection{Final remarks} \label{final-remarks}

It might be interesting to discuss how to extend the above theorem to the general case of abelian varieties. One can 
easily show that if $A/\Q$ is a modular variety of dimension $d$ and conductor $N^d$, and no prime divides $N$ to a 
power greater than $3$, then the argument for elliptic curves developed in the previous section also works for 
abelian varieties. However, it is easy to construct examples in which we do not have Heegner points in any cover of $X_{\cR_\mathrm{min}}$ if we allow the conductor of $A$ to be divisible by arbitrary powers of $p$ if we just consider 
orders of type $(\Ne;\Nc;\{(L_p,\nu_p)\})$, as the following examples show: 

\begin{example} \label{example1}
Let $A/\Q$ be a modular abelian variety of conductor $N^d$, and suppose $N=p^nq$ with $p$ and $q$ distinct 
odd primes and $n=2\varrho+1$ an odd integer. Let $\chi$ be a character of conductor $p^m$ with $m\geq 1$. 
Suppose that $p$ is ramified in and $q$ is inert in $K$. Then $q\in\Sigma(A,\chi)$, and $p\in\Sigma(A,\chi)$ 
if and only if 
\[
\varepsilon_p(A,\chi)=-\eta_{K,p}(-1)=-(-1)^{(p-1)/2}.
\] 
Now assume that $n$ is minimal among the conductor of all twists of $\pi_{A,p}$. In this case 
\cite[Prop. 3.2]{Tunnell} shows that if $m\leq n-1$ then $\varepsilon_p(A,\chi)=-1$, so if 
$p\equiv 3\mod 4$ then $p\in\Sigma(A,\chi)$. If now $m < (n-1)/2$, then comparing with \S \ref{appA} 
we see that there are no Heegner points of conductor $p^m$ in any cover of $X_{\cR_\mathrm{min}}$ 
associated with an order as in Definition \ref{def-ord}.
\end{example}

\begin{example}  \label{example2}
As in the above example, let $A/\Q$ be a modular abelian variety of conductor $N^d$, and suppose now that $N=p^nq$ 
with $p$ and $q$ two odd primes and $n=2\varrho\geq 4$ an even integer. Let $\chi$ be a character of conductor $p^m$ 
with $m\geq 1$. Suppose first that $p\in\Sigma(A,\chi)$, so $\varepsilon_p(A,\chi)=-1$, and $q$ is inert in $K$, so 
$q\in\Sigma(A,\chi)$. Consider the quaternion algebra $B$ of discriminant $pq$ and the order $\cR_\mathrm{min}$ of 
$B$ and form the corresponding Shimura curve $X_{\cR_\mathrm{min}}$. From \S \ref{appA}, we see that if 
$m<n/2$, then there are no Heegner points of conductor $p^m$ in any covering of $X_{\cR_\mathrm{min}}$ associated 
with special orders as in Definition \ref{def-ord}. Secondly, suppose $p\not\in\Sigma(A,\chi)$, so 
$\varepsilon_p(A,\chi)=+1$ and $q$ is split in $K$, so $q\not\in\Sigma(A,\chi)$. In this case, if $m<n/2$ then again 
there are no Heegner points of conductor $p^m$ in any cover of the Shimura curve $X_{\cR_\mathrm{min}}$ associated 
with Eichler orders (which, in this case, correspond to modular curves and usual congruence subgroups). 
\end{example}

As we may see from the above examples, it seems to us that that one should introduce more general type of orders 
to find other sources of Heegner points defined over the predicted ring class field. 

\begin{conjecture}\label{conjecture}
Let $A/\Q$ be a modular abelian variety, $K$ be an imaginary quadratic field, $\chi$ be an anticyclotomic character 
factoring through the ring class field $H_c$ of $K$ of conductor $c\geq 1$, and suppose that the cardinality of 
$\Sigma(A,\chi)$ is odd. Let $B$ denote the indefinite quaternion algebra of discriminant $\Delta$ equal to the product 
of the finite primes in $\Sigma(A,\chi)$, and $f$ be the newform associated with $A$. Suppose that $f$ is primitive 
and has $p$-minimal Artin conductor, for all primes $p$. 
Then, there exists an open compact subgroup $U$ in $\hat B^\times$ equipped with a surjective morphism 
$J_U \twoheadrightarrow A$ and such that the set of Heegner points in $X_{U}(H_c)$ is non-empty.
\end{conjecture}

As a variant of the above conjecture, one can ask if we can take $U=\hat{\mathcal R}^\times$ for some 
global order $\mathcal R$ in $B$. 
This conjecture is inspired by Corollary \ref{coro-final}; we only point out that the relevant part in this conjecture 
is to show the existence of suitable open compact subgroups (non necessarily arising from global orders) so that we have a good understanding of rationality questions of points arising from embeddings $K\hookrightarrow B$. This would allow us to solve cases as \eqref{MissingCase} in the introduction and discuss Examples \ref{example1} and \ref{example2} above.

\providecommand{\bysame}{\leavevmode\hbox to3em{\hrulefill}\thinspace}
\providecommand{\MR}{\relax\ifhmode\unskip\space\fi MR }
\providecommand{\MRhref}[2]{%
  \href{http://www.ams.org/mathscinet-getitem?mr=#1}{#2}
}
\providecommand{\href}[2]{#2}

\end{document}